\documentclass[11pt]{amsart}
\usepackage{amsmath}
\usepackage{amssymb}
\usepackage{amsthm}
\usepackage{amscd}
\usepackage{fontenc}
\usepackage{url}
\usepackage{hyperref}
\usepackage[all]{xy}
\usepackage{graphicx}
\usepackage{tikz}

\numberwithin{equation}{section}

\newcounter{AbcT}

\newtheorem {Theorem}    {Theorem}[section]
\newtheorem* {Theorem1.9}    {Theorem 1.9}
\newtheorem* {Theorem1.6}    {Theorem 1.6}
\newtheorem* {Question1.7}    {Question 1.7}

\newtheorem* {Definition} {Definition} 
\newtheorem* {Remark}	{\bf{Remark}}

\newtheorem {Question}    [Theorem]{Question}

\newtheorem {Lemma}      [Theorem]    {Lemma}
\newtheorem {Corollary}   [Theorem] {Corollary}
\newtheorem {Proposition}[Theorem]    {Proposition}
\newtheorem {Claim}      [Theorem]    {Claim}
\newtheorem {Observation}[Theorem]    {Observation}

\theoremstyle{remark}

\newcounter{DM@bibnum}

\newcommand{\la}{\langle}
\newcommand{\ra}{\rangle}

\def\IA{{\rm IA}}
\def\SL{{\rm SL}}
\def\GL{{\rm GL}}
\def\Sp{{\rm Sp}}

\def\deg{{\rm deg\,}}

\def\Aut{{\rm Aut}}

\def\Hom{{\rm Hom}}

\def\Mod{{\rm Mod}}
\def\Cent{{\rm Cent}}

\def\Ker{{\rm Ker\,}}
\def\Im{{\rm Im\,}}

\def\NSL_2{{\mathcal N SL_2}}

\def\lam{\lambda}            
                
\def\phi{\varphi}

\def\calA{{\mathcal A}}

\def\calG{{\mathcal G}}

\def\calI{{\mathcal I}}

\def\calK{{\mathcal K}}

\def\calM{{\mathcal M}}

\def\calT{{\mathcal T}}

\def\hbar{\bar h}

\def\dbC{{\mathbb C}}

\def\dbF{{\mathbb F}}

\def\dbN{{\mathbb N}}

\def\dbQ{{\mathbb Q}}
\def\dbR{{\mathbb R}}

\def\dbZ{{\mathbb Z}}

\def\dbn{{\mathbf n}}
\def\dbg{{\mathbf g}}

\begin{document}

\title{On finiteness properties of the Johnson filtrations}

\author{Mikhail Ershov}
\address{University of Virginia}
\email{ershov@virginia.edu}
\author{Sue He}
\address{University of Virginia}
\email{suzy.he@gmail.com}

\thanks{The work of the first author was partially supported by the NSF grant DMS-1201452.}

\begin{abstract} Let $\Gamma$ be either the automorphism group of the free group of rank $n\geq 4$ or the mapping class group of an orientable surface of genus $n\geq 12$ with at most $1$ boundary component, and let $G$ be either the subgroup of $\IA$-automorphisms or the Torelli subgroup
of $\Gamma$, respectively. For $N\in\dbN$ denote by $\gamma_N G$ the $N^{\rm th}$ term of the lower central series of $G$.
We prove that 
\vskip .12cm
\begin{itemize}
\item[(i)] any subgroup of $G$ containing $\gamma_2 G=[G,G]$ (in particular, the Johnson kernel in the mapping class group case) is finitely generated;
\item[(ii)] if $N=2$ or $n\geq 8N-4$ and $K$ is any subgroup of $G$ containing
$\gamma_N G$ (for instance, $K$ can be the $N^{\rm th}$ term of the Johnson filtration of $G$), then $G/[K,K]$ is nilpotent and hence the abelianization of $K$ is finitely generated;  
\item[(iii)] if $H$ is any finite index subgroup of $\Gamma$ containing $\gamma_N G$, with $N$ as in (ii), then $H$ has finite abelianization.
\end{itemize}
\end{abstract}
\maketitle

\section{Introduction}

\subsection{The Johnson filtrations}
Let $F_n$ denote the free group of rank $n\geq 2$,
and let $\IA_n\subset \Aut(F_n)$ denote the subgroup of automorphisms of $F_n$ which act as identity on the abelianization; equivalently,
$\IA_n$ is the kernel of the natural map $\Aut(F_n)\to \Aut(F_n^{ab})\cong \GL_n(\dbZ)$. More generally, for each $k\in\dbN$ define $\IA_n(k)$ to be the kernel of the natural map $\Aut(F_n)\to \Aut(F_n/\gamma_{k+1}F_n)$. The filtration $\IA_n=\IA_n(1)\supset \IA_n(2)\supset \ldots$ was first introduced and studied by Andreadakis in \cite{An}. It is easy to see that $\cap_{k=1}^{\infty}\IA_n(k)=\{1\}$ and
$\gamma_k \IA_n\subseteq \IA_n(k)$ for each $k$ (as usual, for a group $G$ and $k\in\dbN$ we denote by $\gamma_k G$ the $k^{\rm th}$ term of the lower central series of $G$ defined inductively by
$\gamma_1 G=G$ and $\gamma_{k+1}G=[\gamma_k G,G]$ for $k\geq 1$). Bachmuth~\cite{Ba} proved that $\gamma_2 \IA_n= \IA_n(2)$. 

The filtration $\{\IA_n(k)\}_{k\in\dbN}$ is often referred to as the {\it Johnson filtration} and owes its name
to the corresponding filtration in mapping class groups whose study was initiated by Johnson~\cite{JoS}.
Let $\Sigma_{g}^1$ be an orientable surface of genus $g\geq 2$ with $1$ boundary component and
$\Mod_{g}^1$ its mapping class group. The fundamental group $\pi=\pi(\Sigma_{g}^1)$ is free of rank $2g$,
and for each $k\in\dbN$ there is a natural homomorphism $\Mod_{g}^1\to \Aut(\pi/\gamma_{k+1}\pi)$. Denote the kernel of this homomorphism by $\calI_g^1(k)$. The subgroups $\calI_g^1=\calI_g^1(1)$ and $\calK_g^1=\calI_g^1(2)$ are known as the {\it Torelli subgroup} and the {\it Johnson kernel}, respectively, and the filtration $\{\calI_g^1(k)\}_{k\in\dbN}$ is called the {\it Johnson filtration}
of $\Mod_{g}^1$. Again one has $\cap_{k=1}^{\infty}\calI_g^1(k)=\{1\}$ and
$\gamma_k \calI_g^1 \subseteq \calI_g^1(k)$ for each $k$, but this time the inclusion is known to be strict already for $k=2$.
The mapping class group $\Mod_g$ of a closed orientable surface of genus $g$
is a quotient of $\Mod_{g}^1$, and the Johnson filtration $\{\calI_g(k)\}_{k\in\dbN}$ of $\Mod_g$ can be defined as the image of the filtration
$\{\calI_g^1(k)\}_{k\in\dbN}$ of $\Mod_g^1$.

\subsection{Finite generation}
The first result on finiteness properties of the Johnson filtrations is a classical theorem of Magnus~\cite{Ma} from 1935 which asserts that $\IA_n=\IA_n(1)$ is finitely generated for all $n\geq 2$. In his celebrated 1983 paper, Johnson~\cite{Jo3} proved that the Torelli group $\calI_g^1$ (and hence also its quotient $\calI_g$) is finitely generated for $g\geq 3$. In 1986, McCullough and Miller~\cite{MM} proved that $\calI_2$ (and hence also $\calI_2^1$) is infinitely generated, thereby completely settling the finite generation question for the first terms. 

The question whether the Johnson kernel $\calI_g(2)$ is finitely generated for $g\geq 3$ has received a lot of attention since 1990s and was explicitly asked by Morita \cite[Problem~2.2(i)]{Mo}. For some time it was generally expected that the Johnson kernel $\calI_g(2)$ as well as its counterpart $\IA_n(2)$ are infinitely generated and moreover  have infinite-dimensional first rational homology. The latter was refuted by the recent breakthrough results of Dimca and Papadima~\cite{DP} and Papadima and Suciu~\cite{PS} who established that  $H_1(\calI_g(2),\dbQ)$ for $g\geq 4$ and
$H_1(\IA_n(2),\dbQ)$ for $n\geq 5$ are finite-dimensional. Our first main theorem strengthens these results by showing that the second terms of the Johnson filtrations are finitely generated in sufficiently large rank. 

\begin{Theorem} 
\label{thm:supermain}
Let $G$ be either $\IA_n$ for $n\geq 4$ or $\calI_g^1$ for $g\geq 12$. If $K$ is any subgroup of $G$ containing $[G,G]$, then $K$ is finitely generated. 
\end{Theorem}

\begin{Remark}\rm Theorem~\ref{thm:supermain} is also true for $G=\calI_g$ since this group is a quotient of $\calI_g^1$.
\end{Remark}

We will prove Theorem~\ref{thm:supermain} by computing the \emph{BNS invariant} of the group $G$ in the theorem.
Given an arbitrary finitely generated group $G$, by a {\it character} of $G$ we will mean a homomorphism
$\chi:G\to\dbR$. Two characters $\chi$ and $\chi'$ will be considered equivalent if they are positive scalar multiples of each other.
The equivalence class of a character $\chi$ will be denoted by $[\chi]$, and the set of equivalence classes of nonzero characters will be denoted by $S(G)$. In \cite{BNS}, Bieri, Neumann and Strebel introduced a certain subset $\Sigma(G)$
of $S(G)$, now known as the {\it BNS invariant of $G$}. The BNS invariant of $G$ completely determines which subgroups of $G$ containing $[G,G]$ are finitely generated:

\begin{Theorem}[\cite{BNS}]
\label{BNS}
Let $G$ be a finitely generated group and let $N$ be a subgroup of $G$ containing $[G,G]$. Then $N$ is finitely generated
if and only if $\Sigma(G)$ contains $[\chi]$ for every non-trivial $\chi$ which vanishes on $N$. In particular, $[G,G]$ is finitely generated if and only if $\Sigma(G)=S(G)$.  
\end{Theorem}
 
Thus, Theorem~\ref{thm:supermain} is a direct consequence of Theorem~\ref{BNS} and the following result:

\begin{Theorem}
\label{OK}
Let $G$ be as in Theorem~\ref{thm:supermain}. Then $\Sigma(G)=S(G)$.
\end{Theorem}

The set $\Sigma(G)$ admits many different characterizations. In order to prove Theorem~\ref{OK} we will use the characterization in terms of actions on real trees due to Brown~\cite{Br} (see \S~\ref{sec:cent}).

\subsection{Vanishing of cohomology and abelianization of the Johnson filtrations}
\label{sec:vanish}

Our second main theorem proves the vanishing of the first cohomology of $\IA_n$ and $\calI_g^1$ with coefficients in a certain class of irreducible representations.

\begin{Theorem}
\label{coh_main} Let $N\geq 2$ be an integer and $G$ a group such that one of the following holds:  
\begin{itemize}
\item[(a)] $N=2$ and $G=\IA_n$ where $n\geq 4$;
\item[(b)] $G=\IA_n$ where $n\geq 8N-4$;
\item[(c)] $G=\calI_g^1$ where $g\geq 8N-4$.
\end{itemize}
Let $\rho$ be a non-trivial irreducible (possibly infinite-dimensional) representation of $G$ over an arbitrary field $F$ such that
$\Ker\rho\supseteq \gamma_N G$ (equivalently, $\rho(G)$ is nilpotent of class less than $N$). Then
$$H^1(G,\rho)=0.$$
\end{Theorem}

\begin{Remark}\rm Theorem~\ref{coh_main} remains true if $G$ is replaced by any of its quotients. Indeed, if $K$ is a normal subgroup of a group $G$
and $\rho$ is any representation of $G/K$, then $H^1(G/K,\rho)$ embeds in $H^1(G,\rho)$, e.g., by the inflation-restriction sequence.
\end{Remark}

Our first application of Theorem~\ref{coh_main} yields new structural information about the abelianization of subgroups of $G$ containing $\gamma_N G$, with $G$ and $N$ as in Theorem~\ref{coh_main}. In particular, the result applies to the $N^{\rm th}$ term of the Johnson filtration of $G$.

\begin{Corollary}
\label{cor:fgab} Let $G$ and $N$ be as in Theorem~\ref{coh_main}, and let $K$ be any subgroup of $G$ containing $\gamma_N G$. Then
the following equivalent conditions hold:
\begin{itemize}
\item[(i)] $G/[K,K]$ is nilpotent (equivalently, $\gamma_m G\subseteq [K,K]$ for some $m$)
\item[(ii)] $K/[K,K]\cong H_1(K,\dbZ)$ considered as a $G/K$-module is annihilated by some power
of the augmentation ideal of $\dbZ[G/K]$.
\end{itemize}
In particular, $K/[K,K]$ is finitely generated.
\end{Corollary}
\begin{Remark}\rm The equivalence of conditions (i) and (ii) above will be explained in the proof of Theorem~\ref{kernel_main}. The last assertion
of Corollary~\ref{cor:fgab} holds since a group $G$ in Corollary~\ref{cor:fgab} is finitely generated (hence so is $G/[K,K]$) and a subgroup of a finitely generated nilpotent group is finitely generated.
\end{Remark}
When $N=2$, Corollary~\ref{cor:fgab} asserts that if $G=\IA_n$ with $n\geq 4$ or $G=\calI_g^1$ with $g\geq 12$, then
any metabelian quotient of $G$ is nilpotent, which we find quite surprising.
\vskip .1cm

Corollary~\ref{cor:fgab} is a direct consequence of Theorem~\ref{coh_main} and the following general result which will be proved in \S~\ref{sec:abel}:

\begin{Theorem} 
\label{kernel_main}
Let $G$ be a finitely generated group, and let $K$ be a normal subgroup of $G$ such that $Q=G/K$ is nilpotent.
Assume that for any finite field $F$ of prime order and any non-trivial finite-dimensional irreducible representation $\rho$
of $G$ over $F$ with $\Ker\rho\supseteq K$ we have $H^1(G,\rho)=0$. Then the quotient $G/[K,K]$ is nilpotent or, equivalently,
$K/[K,K]$ is annihilated by some power of the augmentation ideal of $\dbZ[Q]$.
\end{Theorem}

Theorem~\ref{coh_main} and Corollary~\ref{cor:fgab} generalize two known results (with the extra restriction $g\geq 12$ in the mapping class group case) discussed below. Both these results deal with the case $N=2$ and were originally obtained by rather different methods.

If $\rho$ is a one-dimensional representation of a group $G$, we always have $\Ker\rho\supseteq [G,G]=\gamma_2 G$.
Thus Theorem~\ref{coh_main}(a) implies the following: for any field $F$ and any $n\geq 4$ the set 
$$\{\rho\in \Hom(\IA_n,F^{\times}): H^1(\IA_n,F_{\rho})\neq 0\}$$ has just one element (the trivial homomorphism)\footnote{The trivial homomorphism does belong to this set since $\IA_n$ is a finitely generated group with infinite abelianization 
(see \S~\ref{sec:ianab}).}, where $F_{\rho}$ denotes $F$ with $g\in\IA_n$ acting on $F$ as multiplication by $\rho(g)$. 
In the case $F=\dbC$ this answers in the affirmative \cite[Question~10.6(2)]{PS}. In \cite[Thm~10.5]{PS} it was proved that the above set is finite.\footnote{The question in \cite{PS} is stated in terms of homology, not cohomology, but this does not matter since for any group $G$ and any representation $\rho$ of $G$ over a field there are isomorphisms $H^i(G,\rho^*)\cong H_i(G,\rho)^*$ 
for all $i\in\dbN$ where $\rho^*$ is the dual representation. This can be proved by the argument from \cite[Prop~7.1]{Br2}.
} 
\vskip .12cm
Corollary~\ref{cor:fgab} and Theorem~\ref{coh_main} in the case $G=\calI_g^1, N=2$ generalize a result of 
Dimca, Hain and Papadima~\cite[Thm~A]{DHP}. The latter asserts that $H_1(\calI_g(2),\dbQ)$ is annihilated by some power of the augmentation ideal of $\dbQ[\calI_g/\calI_g(2)]$ and $H^1(\calI_g,\dbC_{\rho})=0$ for any non-trivial $\rho:\calI_g\to\dbC^{\times}$ which vanishes on $\calI_g(2)$, in both cases for all $g\geq 4$. The proofs of both \cite[Thm~A]{DHP} and \cite[Thm~10.5]{PS} made an essential  use of ideas from algebraic geometry introduced in \cite{DP}.
\vskip .12cm

\subsection{Abelianization of finite index subgroups}
Our second application of Theorem~\ref{coh_main} is motivated by the following well-known question (see, e.g., \cite{BV,Iv,PW}):

\begin{Question}
\label{q:FAb}
Let $\Gamma=\Aut(F_n)$ or $\Mod_g^1$. If $H$ is a finite index subgroup of $\Gamma$, does $H$ always have finite abelianization? 
\end{Question}

The answer is known to be negative (for general $H$) for $n=2$ (since $\Aut(F_2)$ projects onto $\GL_2(\dbZ)$, which is a virtually free group), $n=3$ (see \cite{GL} and \cite{Mc}) and $g=2$ (see \cite{Mc1}). For $n\geq 4$ and $g\geq 3$, Question~\ref{q:FAb} is open and appears to be difficult. 
\vskip .1cm
Here is a natural counterpart of Question~\ref{q:FAb} for finite index subgroups of $\IA_n$ or $\calI_g^1$:

\begin{Question}
\label{q:FAb2}
Let $G=\IA_n$ or $\calI_g^1$. If $H$ is a finite index subgroup of $G$, is it true that $\dim H^1(H,\dbC)=\dim H^1(G,\dbC)$? 
\end{Question}
\begin{Remark}\rm If $H$ has finite index in $G$, the restriction map $H^1(G,\dbC)\to H^1(H,\dbC)$ is always injective since the additive group of $\dbC$ is torsion-free. Hence $\dim H^1(H,\dbC)=\dim H^1(G,\dbC)$ if and only if the restriction map $H^1(G,\dbC)\to H^1(H,\dbC)$ is an isomorphism. \end{Remark}

As we will explain in \S~\ref{sec:fabproof}, a positive answer to Question~\ref{q:FAb2} would yield a positive answer to Question~\ref{q:FAb}.
Even if the answer to one or both of those questions ends up being negative (we expect this is the case with Question~\ref{q:FAb2}), it is still interesting to obtain positive results for some classes of subgroups $H$. Theorem~\ref{FAb_concrete} below establishes the latter for subgroups containing $\gamma_N G$ with $G$ and $N$ as in 
Theorem~\ref{coh_main}.

\begin{Theorem}
\label{FAb_concrete}
Let $G$ and $N$ be as in Theorem~\ref{coh_main}, 
and let $\Gamma=\Aut(F_n)$ if $G=\IA_n$
and $\Gamma=\Mod_g^1$ if $G=\calI_g^1$. Let $H$ be a subgroup of $\Gamma$ containing $\gamma_N G$.
The following hold:
\begin{itemize} 
\item [(1)] If $H$ is a finite index subgroup of $G$, then $\dim H^1(H,\dbC)=\dim H^1(G,\dbC)$.
\item [(2)] If $H$ is a finite index subgroup of $\Gamma$, then $H$ has finite abelianization.
\end{itemize}
\end{Theorem}

To the best of our knowledge, the only previous result dealing with Question~\ref{q:FAb2} is a theorem of
Putman~\cite[Thm B]{Pu2} who used a different method to establish Theorem~\ref{FAb_concrete}(1) for $G=\calI_g^1$, with $g\geq 3$, and $H$ a subgroup containing  the Johnson kernel $\calI_g^1(2)$. It is interesting to note that while we will deduce Theorem~\ref{FAb_concrete} from Theorem~\ref{coh_main}, an opposite implication was used in 
\cite{DHP} where  \cite[Thm B]{Pu2} was one of the tools in the proof of \cite[Thm~A]{DHP} stated in \S~\ref{sec:vanish}.

The prior work on Question~\ref{q:FAb} is more extensive and will be summarized in \S~\ref{sec:fabremarks}. 

\subsection{On the proofs of Theorems~\ref{OK} and \ref{coh_main}}

Let $G$ be a finitely generated group. We will show that there is a common approach to solving the following two problems:
\begin{itemize}
\item[(1)] prove that $H^1(G,\rho)=0$ for a given irreducible representation $\rho$ of $G$ and 
\item[(2)] prove that $[\rho]\in\Sigma(G)$ for a given character $\rho$ of $G$.
\end{itemize}
In both cases the desired result follows from the existence of a {\it $\rho$-centralizing generating} (abbreviated as $\rho$-CG)  sequence in $G$.
The notion of a $\rho$-CG sequence, introduced in \S~\ref{sec:cent}, will be defined with respect to an arbitrary homomorphism $\rho$ from $G$ to another group. Without giving a formal definition here, we just mention that a $\rho$-CG sequence exists whenever $G$ has finitely many pairwise commuting elements  $t_1,\ldots, t_k$ whose centralizers generate $G$
and such that $\rho(t_i)$ is non-trivial and central in $\rho(G)$ for each $i$ (see Lemma~\ref{centr_trivial}). Note that by our hypotheses we will always deal with $\rho$ such that $\rho(G)$ is nilpotent (and non-trivial) and hence always contains non-trivial central elements.

\vskip .12cm
In the following two paragraphs we assume that $G$ and $N$ are as in Theorem~\ref{coh_main} and $\rho$ is a non-trivial homomorphism from $G$ to another group  with $\Ker\rho\supseteq \gamma_N G$ (in particular, $\rho$ can be any non-trivial character of $G$). Our goal is to show that $G$ has a $\rho$-CG sequence. 

The group $\IA_n$ has a very simple finite generating set constructed by Magnus already in the 1930s (see \S~\ref{sec:ian} for its definition). A generating set for $\calI_g^1$ which is in many ways analogous to Magnus' generating set for $\IA_n$ was constructed in a recent paper of Church and Putman~\cite{CP}, which made use of an earlier work of Putman~\cite{Pu3} and the original work of Johnson~\cite{Jo3} (see \S~\ref{sec:torelli}.1). We will refer to these generating sets as {\it standard}. We will show that if $\rho$ is ``generic'' (in a suitable sense), then the existence of a $\rho$-CG sequence in $G$ follows from the basic relations between the standard generators (see Lemma~\ref{lem:key}). For a general $\rho$, we will use the observation that $G$ is a normal subgroup of $\Gamma$, with $\Gamma=\Aut(F_n)$ or $\Gamma=\Mod_g^1$, and precomposing $\rho$ by the conjugation by a fixed element of $\Gamma$ does not affect the existence of a $\rho$-CG sequence. 

Thus we are reduced to showing that an arbitrary $\rho$ as above can be conjugated to a generic one. In the case $G=\IA_n$ and $N=2$ this can be done quite explicitly (see \S~\ref{sec:ian}). In the other cases an explicit calculation may be possible, but would be cumbersome, so we take a more conceptual approach. First it is not hard to see that the relevant information about the conjugation action of $\Gamma$ on $G$ is captured by
the induced action of $\Gamma/G$ on $L(G)=\bigoplus\limits_{n=1}^{\infty}\gamma_n G/\gamma_{n+1}G$, the Lie algebra of $G$ with respect to the lower central series. Second, in both cases $(\Gamma,G)=(\Aut(F_n),\IA_n)$ and $(\Mod_n^1,\calI_n^1)$, the quotient $\Gamma/G$ contains a natural copy of $\SL_n(\dbZ)$ which therefore acts on $L(G)$. 

In \S~\ref{sec:maintech} we will establish a sufficient condition on an abstract group $G$
(assuming $\SL_n(\dbZ)$ acts on $L(G)$ as above) which guarantees that $G$ has a $\rho$-CG sequence for any non-trivial 
homomorphism $\rho$ from $G$ to another group such that $\rho(G)$ is nilpotent and has sufficiently small nilpotency class (see Theorem~\ref{coh_main_new}). One of the hypotheses in Theorem~\ref{coh_main_new} is that the abelianization $G^{ab}=G/[G,G]$ is a {\it regular $\SL_n(\dbZ)$-module}, a technical notion which will be introduced in \S~\ref{sec:reg}. While the proof of Theorem~\ref{coh_main_new} will deal with the action of $\SL_n(\dbZ)$ on the entire Lie algebra $L(G)$, its hypotheses only involve the action on $G^{ab}$, the degree $1$ component
of $L(G)$. This is very important since while the structure of $G^{ab}$ as an  $\SL_n(\dbZ)$-module is completely understood and easy to describe for $G=\IA_n$ or $\calI_n^1$, this is not the case with higher degree components of $L(G)$.

\subsection{Some closing comments}
\label{sec:closingcomments}
In a sequel paper \cite{CEP}, Theorem~\ref{thm:supermain} will be generalized to higher terms of the lower central series; more precisely, it will be shown that any subgroup of $\IA_n$ containing $\gamma_N \IA_n$ is finitely generated whenever $n\geq 4N-3$ and any subgroup of $\calI_n^1$ containing $\gamma_N \calI_n^1$ is finitely generated whenever $n\geq \max\{4,2N-1\}$. This result obviously implies the last assertion of Corollary~\ref{cor:fgab} (but not the full statement). The method from \cite{CEP} does not seem to be applicable to Theorem~\ref{coh_main}.

The question whether $[G,G]$ is finitely generated for $G=\IA_n$ or $\calI_n^1$ is now settled for all $n\neq 3$. The answer is positive for $n\geq 4$
(by Theorem~\ref{thm:supermain} and \cite{CEP}) and negative for $n=2$. The latter holds since $\IA_2$ is a free group of rank $2$ (see \S~\ref{sec:ian})
and $\calI_2$ is a free group of countable rank, as proved in \cite{Me}. It seems likely that the group $[\IA_3,\IA_3]$ is not finitely generated and moreover has infinite-dimensional first rational homology. Some evidence for this is provided by the fact that the assertions of Theorem~\ref{coh_main} and both parts of Theorem~\ref{FAb_concrete} are false for $G=\IA_3$ and $N=2$ (see Proposition~\ref{prop:ia3}).

\vskip .2cm
{\bf Organization}.  
In \S~\ref{sec:cent} we will introduce the notion of a $\rho$-CG sequence and reduce Theorems~\ref{OK} and \ref{coh_main} to the problem of existence of $\rho$-CG sequences for suitable $\rho$ (see Proposition~\ref{cent_exist}). In \S~\ref{sec:ian} we will prove that $G=\IA_n$, $n\geq 4$,
admits a $\rho$-CG sequence for any $\rho$ with $\rho(G)$ abelian (see Theorem~\ref{thm:degree2}). Theorem~\ref{thm:degree2} and Proposition~\ref{cent_exist} immediately imply Theorem~\ref{OK} (and hence Theorem~\ref{thm:supermain}) for $G=\IA_n$ and Theorem~\ref{coh_main}(a).

In \S~\ref{sec:ngps} we will introduce the notion of a {\it partially commuting $\dbn$-group},  where $\dbn=\{1,2,\ldots, n\}$, which formalizes the key properties of basic commutation relations in $\Aut(F_n)$ and $\Mod_g^1$. Then in \S~\ref{sec:proofsketch} we will describe a general method of constructing $\rho$-CG sequences in partially commuting $\dbn$-groups. In \S~\ref{sec:reg} we will introduce and study the notion of a regular $\SL_n(\dbZ)$-module. In \S~\ref{sec:maintech}
we will state and prove our main technical theorem (Theorem~\ref{coh_main_new}). Combined with Proposition~\ref{cent_exist}, Theorem~\ref{coh_main_new}
will be used to establish the remaining cases of Theorems~\ref{OK} and \ref{coh_main}. The fact that the hypotheses of Theorem~\ref{coh_main_new} hold for
$G=\IA_n$ is easy and will checked right after the statement of Theorem~\ref{coh_main_new}. The corresponding verification for $G=\calI_g^1$ is more involved and will be made in \S~\ref{sec:torelli}. 

In \S~\ref{sec:abel} we will prove Theorem~\ref{kernel_main}, thereby completing the proof of Corollary~\ref{cor:fgab}. Finally, in \S~\ref{sec:fab} we will prove Theorem~\ref{FAb_concrete} and also discuss other known results of the same kind. These two sections do not use any technical
machinery introduced in the paper and can be read right after the introduction.

\vskip .2cm
{\bf Acknowledgments}. We are extremely grateful to Andrei Rapinchuk who encouraged us to take on this project and suggested 
the general approach to the proof of Theorem~\ref{FAb_concrete}. We are also indebted to Andrei Jaikin who proposed a tremendous simplification of our original proof of Theorem~\ref{kernel_main}. We are very grateful to the anonymous referee who carefully read the paper and made a number of suggestions that helped improve the exposition. Finally, we would like to thank Thomas Church, Andrew Putman and Zezhou Zhang for useful comments on earlier versions of this paper.

\vskip .2cm
{\bf Notation}. When considering cohomology of a group $G$ with coefficients in some representation, we will sometimes use the notation $H^k(G,V)$ where $V$ is the underlying space of the representation and sometimes $H^k(G,\rho)$ when $\rho$ is the action of
$G$ on $V$. We hope that this inconsistency will not cause confusion.

\section{$\rho$-centralizing generating sequences}
\label{sec:cent}

Let $G$ be a group. Given a homomorphism $\rho$ from $G$ to another group, denote by $C_{\rho}(G)$ the set
of all $g\in G$ such that $\rho(g)$ is a non-trivial central element of $\rho(G)$. In particular, if $\rho(G)$ is abelian, then $C_{\rho}(G)=G\setminus\Ker\rho$.

\begin{Definition}\rm Let $g_1,g_2,\ldots, g_k$ be a finite sequence of elements of $G$, and for each $1\leq i\leq k$ let $G_i=\la g_j: 1\leq j\leq i\ra$ be the subgroup generated by the first $i$ elements of the sequence. We will say that the sequence $\{g_n\}$ is {\it $\rho$-centralizing} if $g_1\in C_{\rho}(G)$ and for each $2\leq i\leq k$ there exists $j<i$ such that $g_j\in C_{\rho}(G)$ and $[g_i,g_j]\in G_{i-1}$.
\end{Definition}

A key problem we will consider in this paper is establishing the existence of a $\rho$-centralizing generating sequence for a given $\rho$
(that is, a $\rho$-centralizing sequence whose elements generate $G$). Since this condition will arise very frequently, we will abbreviate 
{\it $\rho$-centralizing generating} as $\rho$-CG for the rest of the paper.

The following lemma describes a particularly simple way to produce $\rho$-CG sequences.

\begin{Lemma}
\label{centr_trivial} Let $G$ be a finitely generated group, and let $\rho$ be a homomorphism from $G$ to another group.
Suppose that there exist pairwise commuting elements $g_1,\ldots, g_k\in C_{\rho}(G)$ such that the union of their centralizers
$\cup_{i=1}^k \Cent_G(g_i)$ generates $G$. Then $G$ has a $\rho$-CG sequence.  
\end{Lemma}
\begin{proof} Since $G$ is finitely generated, we can find a finite generating set $h_1,\ldots, h_m$ for $G$ contained in $\cup \Cent_G(g_i)$.
It is clear that the sequence $g_1,\ldots, g_k, h_1,\ldots, h_m$ is $\rho$-CG.
\end{proof}

In this section we will show that 
\begin{itemize}
\item[(i)] if $\rho$ is an irreducible representation of $G$ over some field, then the existence of a  $\rho$-CG sequence implies that $H^1(G,\rho)=0$; 
\item[(ii)] if $\rho:G\to\dbR$ is a non-trivial character, the existence of a $\rho$-CG sequence implies that $[\rho]$ lies in $\Sigma(G)$, the BNS invariant of $G$.
\end{itemize}

We will start with briefly recalling Brown's characterization of $\Sigma(G)$ in terms of actions on $\dbR$-trees \cite{Br}
and establishing a very simple property of group $1$-cocycles.

We will need only very basic properties of group actions on $\dbR$-trees (see, e.g. \cite{AB} for the proofs).
Let $X$ be an $\dbR$-tree with metric $d$, and suppose that a group $G$ acts on $X$ by isometries; denote the action by $\alpha$. For each $g\in G$ define $l(g)=\inf_{x\in X}d(\alpha(g)x,x)$ (it is known that the infimum is always attained), and let $\calA_g=\{x\in X: d(\alpha(g)x,x)=l(g)\}$.
An element $g\in G$ is called {\it elliptic} (with respect to $\alpha$) if $l(g)=0$, in which case $\calA_g$ is the fixed point set of $g$. If $l(g)>0$, then $g$ is called {\it hyperbolic}. In this case $\calA_g$ is a line, called {\it the axis of $g$}, and $g$ acts on $\calA_g$ as a translation by $l(g)$; moreover, $\calA_g$ is the unique line invariant under the action of $g$.
An action $\alpha$ is called 
\begin{itemize}
\item{\it abelian} if there exists a character $\chi:G\to\dbR$ such that $l(g)=|\chi(g)|$ for all $g\in G$. If such $\chi$
exists, it is unique up to sign, and we say that $\alpha$ is associated with $\chi$.
\item {\it non-trivial} if it has no (global) fixed points and no (global) invariant line.
\end{itemize}

The following characterization of the BNS invariant $\Sigma(G)$ was obtained by Brown~\cite{Br}:
 
\begin{Theorem} 
\label{thm:Brown}
Let $G$ be a finitely generated group, and let $\chi:G\to\dbR$ be a non-trivial character. Then 
$[\chi]\in \Sigma(G)$ if and only if $G$ has no non-trivial abelian actions associated with $\chi$. In particular, by Theorem~\ref{BNS},
$[G,G]$ is finitely generated if and only if $G$ has no non-trivial abelian actions.
\end{Theorem}

Now let $(\rho,V)$ be a representation of a group $G$. Recall that $H^1(G,\rho)\cong Z^1(G,\rho)/B^1(G,\rho)$ where
\begin{align*}
&Z^1(G,\rho)=\{f: G\to V: f(xy)=f(x)+\rho(x)f(y)\mbox{ for all }x,y\in G\};\\ 
&B^1(G,\rho)=\{f: G\to V: \mbox{ there exists }v\in V
\mbox{ such that } f(x)=\rho(x)v-v\mbox{ for all }x\in G\}.
\end{align*}
  
Also recall that $C_{\rho}(G)=\{g\in G\mbox{ such that } \rho(g) \mbox{ is non-trivial and central in } \rho(G)\}$.  
  
\begin{Lemma}
\label{lem1}
Let $G$ be a group, let $(\rho,V)$ be a non-trivial irreducible representation of $G$ over a field, let $g\in C_{\rho}(G)$, and let $f\in Z^1(G,\rho)$.
The following hold:
\begin{itemize}
\item[(i)] There exists a coboundary $b\in B^1(G,\rho)$ with $f(g)=b(g)$.
\item[(ii)] Suppose that $f(g)=0$. Then $f(x)=0$ for every $x\in G$ such that $f(gx)=f(xg)$; in particular,
$f(x)=0$ for every $x\in G$ which commutes with $g$.
\end{itemize}
\end{Lemma}
\begin{proof} (i) Since $g\in C_{\rho}(G)$ and $\rho$ is irreducible, by Schur's Lemma
the operator $\rho(g)-1$ is invertible. Thus,
if $v=(\rho(g)-1)^{-1}f(g)$, then the map $b(x)=\rho(x)v-v$ is a coboundary with the required property.

(ii) We have $f(xg)=f(x)+\rho(x)f(g)=f(x)$ and $f(gx)=f(g)+\rho(g)f(x)=\rho(g) f(x)$, so 
$(\rho(g)-1)f(x)=0$ and hence $f(x)=0$.
\end{proof}

\begin{Proposition} 
\label{cent_exist}
Let $G$ be a finitely generated group. The following hold:
\begin{itemize}
\item[(a)] If $(\rho,V)$ is an irreducible representation of $G$ over some field and $G$ admits a $\rho$-CG sequence, then $H^1(G,\rho)=0$. 
\item[(b)] Let $\rho:G\to\dbR$ be a non-trivial character.
If $G$ admits a $\rho$-CG sequence, then $[\rho]\in \Sigma(G)$. In particular, if $G$ admits a
$\rho$-CG sequence for every non-trivial character $\rho$, then $[G,G]$ is finitely generated.
\end{itemize}
\end{Proposition} 
\begin{Remark}\rm
A weaker version of Proposition~\ref{cent_exist}(b) has been proved before (see \cite[Lemma~1.9]{KMM}). That lemma covers the case when
$G$ has a generating sequence $g_1,\ldots, g_t$ such that $g_1\in C_{\rho}(g)$ and for each $i\geq 2$ there exists $j<i$ such that $g_i$ and $g_j$ commute and $g_j\in C_{\rho}(g)$.
\end{Remark}

\begin{proof}
In both parts of the proof we let $g_1,\ldots, g_k$ be a $\rho$-CG sequence
and $G_i=\la g_1,\ldots, g_i\ra$ for $1\leq i\leq k$ (so that $G_k=G$). 

(a) Take any $f\in Z^1(G,\rho)$. By Lemma~\ref{lem1}(i), after modifying $f$ by a coboundary, we can assume that $f(g_1)=0$.
Then $f=0$ on $G_1$, and we will now prove that $f=0$ on $G_i$ for all $1\leq i\leq k$ by induction on $i$.

Take $i\geq 2$, and assume that $f=0$ on $G_{i-1}$. By hypothesis there exists $j<i$ and $r\in G_{i-1}$ such
that $g_j\in C_{\rho}(G)$ and $g_i g_j=g_j g_i r$. Then $f(g_i g_j)=f(g_j g_i)+\rho(g_j g_i) f(r)=f(g_j g_i)$, whence
$f(g_i)=0$ by Lemma~\ref{lem1}(ii). Since $G_i=\la G_{i-1},g_i\ra$, we have $f(G_i)=0$ as desired.

(b) The following argument is similar to the proof of the main theorem in \cite{OK}.
Let $(\alpha,X)$ be an abelian action of $G$
on an $\dbR$-tree $X$ associated to $\rho$.
By Theorem~\ref{thm:Brown}, it suffices to show that such $\alpha$ is trivial. 

First note that $C_{\rho}(G)=G\setminus \Ker \rho$ (since $\rho(G)$ is abelian). Thus, if $l:G\to \dbR_{\geq 0}$ is the length function associated to $\alpha$, then  $C_{\rho}(G)=\{g\in G: l(g)>0\}$. In other words, $C_{\rho}(G)$ is precisely the set of hyperbolic elements (with respect to $\alpha$); in particular, $g_1$ is hyperbolic by assumption. We claim that its axis $\calA=\calA_{g_1}$ is invariant under the entire group $G$ (which will finish the proof). We will prove that $\calA$ is invariant under $G_i$ for all $1\leq i\leq k$ by induction on $i$.

Take $i\geq 2$, and assume that $\calA$ is invariant under $G_{i-1}$ (in particular, $\calA$ is the axis for any hyperbolic element in $G_{i-1}$ by the uniqueness of the invariant line of a hyperbolic element). By hypothesis there exists $j<i$ and $r\in G_{i-1}$ such that $g_j$ is hyperbolic and $g_i^{-1} g_j g_i= g_j r$. The element $g_i^{-1} g_j g_i$ is also hyperbolic (being a conjugate of $g_j$), and its axis is 
$\alpha(g_i^{-1})(\calA)$. On the other hand $g_i^{-1} g_j g_i=g_j r\in G_{i-1}$, so $\calA$ is the axis of $g_i^{-1} g_j g_i$. Thus, $\alpha(g_i^{-1})(\calA)=\calA$, so $\calA$ is $g_i$-invariant and hence $G_i$-invariant.   
\end{proof}

\section{Proofs of Theorem~\ref{OK} for $G=\IA_n$ and Theorem~\ref{coh_main}(a)}
\label{sec:ian}

In this section we will prove Theorem~\ref{OK} for $G=\IA_n$ and Theorem~\ref{coh_main}(a).
By Proposition~\ref{cent_exist}, we are reduced to proving the following result:

\begin{Theorem} 
\label{thm:degree2}
Let $G=\IA_n$ for some $n\geq 4$, and let $\rho$ be a non-trivial homomorphism from $G$ to another group with $\Ker\rho\supseteq [G,G]$. 
Then $G$ has a $\rho$-CG sequence.
\end{Theorem}

The proof of Theorem~\ref{thm:degree2} will illustrate many key ideas used in the proof of other cases of Theorems~\ref{OK} and Theorem~\ref{coh_main}
without introducing the general technical machinery.

For the rest of the paper, given $n\in\dbN$, we set $$\dbn=\{1,\ldots,n\}.$$
In \cite{Ma}, Magnus proved that $\IA_n$ is generated by the automorphisms $K_{ij}$, $i\neq j$ and $K_{ijk}$, $i,j,k$ distinct, 
with $i,j,k\in\dbn$, defined below (see also \cite{BBM} for another proof):
$$K_{ij}:\left\{\begin{array}{l}x_i\mapsto x_j^{-1}x_i x_j\\
x_k\mapsto x_k\mbox{ for }k\neq i\end{array}
\right.,\qquad
K_{ijk}:\left\{\begin{array}{l}x_i\mapsto x_i[x_j,x_k]\\
x_l\mapsto x_l\mbox{ for }l\neq i
\end{array}\right.
.
$$
Since $K_{ijk}^{-1}=K_{ikj}$, it is enough to consider $K_{ijk}$ with $j<k$. The elements $K_{ij}$ and $K_{ijk}$ with $j<k$ will be referred to as the {\it standard generators} of $\IA_n$.

\begin{Remark}\rm If $n=2$, there are no generators $K_{ijk}$, and $K_{12}$ and $K_{21}$ are simply the conjugations by $x_2$ and $x_1$, respectively. 
Hence the group $\IA_2$ coincides with the group of inner automorphisms of $F_2$ and thus is free of rank $2$.
\end{Remark}

Below is the list of relations between the elements $K_{ij}$ and $K_{ijk}$ that will be used in the proof of Theorem~\ref{thm:degree2}.

\begin{Lemma} 
\label{lem:rel}
Let $a_1,a_2,a_3,b_1,b_2,b_3\in\dbn$, and assume that $\{a_i\}$ are distinct and $\{b_i\}$ are distinct.
The following relations hold:
\begin{itemize}
\item[(R1)] $[K_{a_1 a_2},K_{b_1 b_2}]=1$ if $a_1\neq b_1,b_2$ and $b_1\neq a_1, a_2$;
\item[(R2)] $[K_{a_1 a_2},K_{b_1 b_2 b_3}]=1$ if $a_1\neq b_1,b_2, b_3$ and $b_1\neq a_1, a_2$;
\item[(R3)] $[K_{a_1 a_2 a_3},K_{b_1 b_2 b_3}]=1$ if $a_1\neq b_1, b_2, b_3$ and $b_1\neq a_1, a_2,a_3$;
\item[(R4)] $[K_{bcd},K_{ab}] = [K_{ad}, K_{ac}]$ if $a,b,c,d\in\dbn$ are distinct.
\end{itemize}
\end{Lemma}

The main technical result we shall establish in this section is the following proposition:

\begin{Proposition}
\label{prop:keyIAn}
Let $G=\IA_n$ for some $n\geq 4$, and let $H$ be a proper subgroup of $G$ containing $[G,G]$. Then there exists $a\in\Aut(F_n)$ such that
the conjugate subgroup $a H a^{-1}$ does not contain any of the elements $K_{12}, K_{21}, K_{34}$ and $K_{43}$.
\end{Proposition}

We will start by deducing Theorem~\ref{thm:degree2} from Proposition~\ref{prop:keyIAn}.

\begin{proof}[Proof of Theorem~\ref{thm:degree2}]
Given $a\in \Aut(F_n)$, define the homomorphism $\rho_{a}$ by $\rho_{a}(x)=\rho(a^{-1}x a)$.
Clearly, if $x_1,\ldots, x_t$ is a $\rho_{a}$-CG sequence in $G$, then $a^{-1}x_1 a, \ldots, a^{-1}x_t a$ is a $\rho$-CG sequence, so it is enough to construct a $\rho_{a}$-CG sequence (for some $a$). Since $\Ker\rho_{a}=a (\Ker\rho)a^{-1}$,
by letting $a$ be an element from the conclusion of Proposition~\ref{prop:keyIAn} with $H=\Ker\rho$ and replacing $\rho$ by $\rho_{a}$, we may assume that $\Ker\rho$ does not intersect the set $S=\{K_{12}, K_{21}, K_{34},K_{43}\}$. Since $\rho(G)$ is abelian, we conclude that $S\subseteq C_{\rho}(G)$.

Using relations (R1)-(R3) of 
Lemma~\ref{lem:rel}, it is easy to show that the only standard generators of $\IA_n$ which do not commute with an element of $S$ are $K_{134}, K_{234}, K_{312}, K_{412}$. However, by relations (R4)
each of these 4 elements commutes with an element of $S$ modulo the subgroup generated by all $K_{xy}$.
Thus, if we order the standard generators of $\IA_n$ so that $K_{12}, K_{34}, K_{21}, K_{43}$ come first (in this order)
and $K_{134}, K_{234}, K_{312}, K_{412}$ come last, then these generators form a $\rho$-CG sequence.
\end{proof}

\subsection{The abelianization of $\IA_n$}
\label{sec:ianab}

Since the quotient $\Aut(F_n)/\IA_n$ is isomorphic to $\GL_n(\dbZ)$, the abelianization $\IA_{n}^{ab}=\IA_n/[\IA_n,\IA_n]$ is a $\GL_n(\dbZ)$-module. 
Let $V=\dbZ^n$, considered as a standard $\GL_n(\dbZ)$-module, and let  $V^*=\Hom(V,\dbZ)$ be the dual module. It is well known that $\IA_{n}^{ab}$ is canonically isomorphic to $V^*\otimes (V\wedge V)$ as a $\GL_n(\dbZ)$-module; to the best of our knowledge, the first full proof of this isomorphism was given by Formanek~\cite{Fo} (for alternative references see, e.g., \cite[Theorem~6.1]{Kaw} or \cite[\S~5]{Bh}).

Let $e_1,\ldots, e_n$ be the standard basis of $V$,
and let $e_1^*,\ldots, e_n^*$ be the corresponding dual basis of $V^*$, and set $$\kappa_{ijk}=e_i^*\otimes (e_j\wedge e_k).$$
The above isomorphism between $\IA_n^{ab}$ and 
$V^*\otimes (V\wedge V)$ is given by 
\begin{equation}
\label{eq:kappa}
[K_{ij}]\mapsto\kappa_{iij}\mbox{ and }[K_{ijk}]\mapsto\kappa_{ijk}
\end{equation}
where $[x]$ denotes the image of $x\in \IA_n$ in $\IA_n^{ab}$. 

From now on we will identify $\IA_n^{ab}$ with $V^*\otimes (V\wedge V)$ via the map \eqref{eq:kappa}.
\vskip .12cm

Given $i,j\in\dbn$ with $i\neq j$, define $E_{ij}, F_{ij}\in \SL_n(\dbZ)$ by
\begin{equation}
\label{eij_def}
E_{ij}:\left\{\begin{array}{l}e_j\mapsto e_j+ e_i\\
e_k\mapsto e_k\mbox{ for }k\neq j\end{array}
\right.\qquad
F_{ij}:\left\{\begin{array}{l}e_i\mapsto -e_j\\
e_j\mapsto e_i\\
e_k\mapsto e_k\mbox{ for }k\neq i,j\end{array}
\right.
\end{equation}
Note that $F_{ij}=E_{ij}E_{ji}^{-1}E_{ij}$.

The following observation describes the action of $E_{ij}^{\pm 1}$ and $F_{ij}$ on the generators
$\kappa_{abc}$ of $\IA_n^{ab}$.
\begin{Observation}
\label{kappa_action} 
$\empty$
\begin{itemize}
\item [(a)] The following hold, where distinct letters denote distinct integers in $\dbn$:
\begin{align*}
&E_{ij}^{\pm 1}\kappa_{iic}-\kappa_{iic}= \mp\kappa_{jic}& &E_{ij}^{\pm 1}\kappa_{jji}-\kappa_{jji}= 0&\\
&E_{ij}^{\pm 1}\kappa_{ibc}-\kappa_{ibc}= \mp\kappa_{jbc}& &E_{ij}^{\pm 1}\kappa_{aaj}-\kappa_{aaj}= \pm\kappa_{aai}&\\
&E_{ij}^{\pm 1}\kappa_{ajc}-\kappa_{ajc}= \pm\kappa_{aic}& &E_{ij}^{\pm 1}\kappa_{iij}-\kappa_{iij}= \pm\kappa_{jji}&\\
&E_{ij}^{\pm 1}\kappa_{jjc}-\kappa_{jjc}= \pm\kappa_{jic}& &E_{ij}^{\pm 1}\kappa_{ijc}-\kappa_{ijc}= \pm(\kappa_{iic}-\kappa_{jjc})-\kappa_{jic}.&
\end{align*} 
\item [(b)] Let $i,j,a,b,c\in\dbn$ with $i\neq j$ and $b\neq c$. Then
\begin{itemize}
\item [(i)] $E_{ij}\kappa_{abc}=\kappa_{abc}$ if $i\neq a$ and $j\neq b,c$;
\item [(ii)] $F_{ij}\kappa_{abc}=\pm \kappa_{\sigma(a)\sigma(b)\sigma(c)}$ where $\sigma$ is the transposition $(i,j)$.
\end{itemize}
\end{itemize}
\end{Observation}
\begin{proof}
The action of $E_{ij}$ and $F_{ij}$ on the basis $\{e_k^*\}$ on $V^*$ is given by
\begin{align*}
&E_{ij}:\left\{\begin{array}{l}e_i^*\mapsto e_i^*- e_j^*\\
e_k^*\mapsto e_k^*\mbox{ for }k\neq i\end{array}
\right.&
&F_{ij}:\left\{\begin{array}{l}e_i^*\mapsto -e_j^*\\
e_j^*\mapsto e_i^*\\
e_k^*\mapsto e_k^*\mbox{ for }k\neq i,j\end{array}
\right.&
\end{align*}
All parts of Observation~\ref{kappa_action} immediately follow from these formulas and \eqref{eij_def}. For instance,
$$E_{ij}^{\pm 1}\kappa_{iic}=E_{ij}^{\pm 1}(e_i^*\otimes (e_i\wedge e_c))=(e_i^*\mp e_j^*)\otimes (e_i\wedge e_c)=
\kappa_{iic}\mp \kappa_{jic}.$$ 
\end{proof}

\subsection{Proof of Proposition~\ref{prop:keyIAn}} 
\label{sec:proofian}
Let $U$ be any abelian group which contains an isomorphic copy of every finitely generated abelian group, and let
$$\Lambda=\Hom_{\dbZ}(\IA_n^{ab},U).$$
We will define the action of $\GL_n(\dbZ)$ on $\Lambda$ in the usual way:
$$(g\lam) (x)=\lam(g^{-1}x)\mbox{ for all }g\in \GL_n(\dbZ), \lam\in\Lambda\mbox{ and }x\in \IA_n^{ab}.$$
Our first goal is to reformulate Proposition~\ref{prop:keyIAn} in terms of the $\GL_n(\dbZ)$-module $\Lambda$ (see Lemma~\ref{lemkey4} below).
\vskip .1cm
Since $\IA_n^{ab}$ (and any of its quotients) is a finitely generated abelian group, 
every subgroup of $\IA_n^{ab}$ is the kernel of some element of $\Lambda$. 
Fix $H$, a proper subgroup of $\IA_n$ containing $[\IA_n,\IA_n]$, and choose $\lam\in \Lambda$
such that $\Ker\lam=H/[\IA_n,\IA_n]$. Note that $\lam\neq 0$ since $H$ is proper.

Recall that for an element $x\in \IA_n$ we denote by $[x]$ its image in $\IA_n^{ab}$. Given $a\in \Aut(F_n)$, let $\overline a$ be its image in $\GL_n(\dbZ)$.

\begin{Claim}
\label{claim:triv} Let $a\in \Aut(F_n)$ and $x\in \IA_n$.
Then $x\in a H a^{-1}$ if and only if $(\overline a\lam)([x])=0$.
\end{Claim}

\begin{proof} 
By definition $(\overline a\lam)([x])=\lam({\overline a}^{-1}[x])=\lam([a^{-1} x a]).$ Thus, 
$$(\overline a\lam)([x])=0 \iff [a^{-1} x a]\in\Ker\lam=H/[\IA_n,\IA_n] \iff a^{-1} x a \in H \iff x\in aH a^{-1}.$$
\end{proof}
Given $i,j,k\in\dbn$ with $j\neq k$ define 
$$c_{ijk}(\lam)=\lam(\kappa_{ijk}).$$
In view of Claim~\ref{claim:triv} and \eqref{eq:kappa}, Proposition~\ref{prop:keyIAn} is equivalent to the following lemma:

\begin{Lemma}
\label{lemkey4}
Assume that $n \geq 4$. Then for any nonzero $\lam\in\Lambda$ there exists $g\in \GL_n(\dbZ)$ such that
$c_{112}(g\lam), c_{221}(g\lam), c_{334}(g\lam)$ and $c_{443}(g\lam)$ are all nonzero.
\end{Lemma}

The following two results follow directly from Observation~\ref{kappa_action}. We will use Observation~\ref{obs:coeff} in 
the proof of Lemma~\ref{lemkey4}  without an explicit reference.

\begin{Observation}
\label{obs:coeff}
Let $i,j,a\in\dbn$ be distinct and $\lam\in \Lambda$. Then
\begin{align*}
&c_{aaj}(E_{ij}\lambda) = c_{aaj}(\lambda) - c_{aai}(\lambda)&
&c_{iij}(E_{ij}\lambda)  = c_{iij}(\lambda) - c_{jji}(\lambda)&\\
& c_{jja}(E_{ij}\lambda) = c_{jja}(\lambda) - c_{jia}(\lambda)&
& c_{iia}(E_{ij}\lambda)  = c_{iia}(\lambda) + c_{jia}(\lambda)&
\end{align*}
and $c_{xxy}(E_{ij}\lambda)=c_{xxy}(\lambda)$ for all $(x,y)\neq (a,j), (i,j),(j,a),(i,a)$.
\end{Observation}

\begin{Observation}
\label{obs:coeff2} Let $i,j\in\dbn$ with $i\neq j$ and let $\sigma$ be the transposition $(i,j)$. Then
for any $\lam\in \Lambda$ and $x,y,z\in\dbn$ we have $c_{xyz}(F_{ij}\lambda)=\pm c_{\sigma(x)\sigma(y)\sigma(z)}(\lam)$.
\end{Observation}

\begin{proof}[Proof of Lemma~\ref{lemkey4}] 
\underline{Step 1}: There exists $g_1\in \SL_n(\dbZ)$ such that $c_{112}(g_1\lambda)\neq 0$.

First of all, by Observation~\ref{obs:coeff2} it suffices to find distinct $a,b\in\dbn$
such that $c_{aab}(g_1\lam)\neq 0$. 

If $c_{xxy}(\lambda) \neq 0$ for some $x\neq y$, we are done. 
Suppose now that $c_{xxy}(\lambda)= 0$ for all
$x\neq y$. Then there must exist distinct distinct $a,b,c$ with $c_{abc}(\lambda)\neq 0$, in which case 
$$c_{aab}(E_{ca}\lambda) = c_{aab}(\lambda) - c_{acb}(\lambda) = 0 + c_{abc}(\lambda) \neq 0.$$

\underline{Step 2}: There exists $g_2\in \SL_n(\dbZ)$ such that $c_{112}(g_2\lambda),c_{113}(g_2\lambda)\neq 0$.

By Step 1, we can assume that $c_{112}(\lambda)\neq 0$.
If $c_{113}(\lambda)\neq 0$, we are done. And if $c_{113}(\lambda)=0$, then
$c_{113}(E_{23}\lambda)= c_{113}(\lambda) - c_{112}(\lambda) = 0 - c_{112}(\lambda)  \neq 0$  
and $c_{112}(E_{23}\lambda) = c_{112}(\lambda) \neq 0.$\\

\underline{Step 3}: There exists $g_3\in \SL_n(\dbZ)$ such that either $c_{112}(g_3\lambda), c_{113}(g_3\lambda), c_{221}(g_3\lambda)\neq 0$ or $c_{112}(g_3\lambda), c_{334}(g_3\lambda)\neq 0$.

By Step 2, we can assume that $c_{112}(\lambda),c_{113}(\lambda)\neq 0$. If 
$c_{224}(\lambda)\neq 0$ or $c_{221}(\lambda)\neq 0$, we are done (using Observation~\ref{obs:coeff2} in the former case),
so assume that $c_{224}(\lambda)=c_{221}(\lambda)= 0$. We consider 3 cases.\\

\underline{Case 1}: $c_{223}(\lambda) \neq 0$.
Then \[ c_{224}(E_{34}\lambda) = c_{224}(\lambda) - c_{223}(\lambda) = -c_{223}(\lambda) \neq 0;\]
\[ c_{113}(E_{34}\lambda) = c_{113}(\lambda) \neq 0.\]
Hence $c_{113}(E_{34}\lambda), c_{224}(E_{34}\lambda) \neq 0$.\\

\underline{Case 2}: $c_{223}(\lambda)= 0$ and $c_{123}(\lambda)= 0$.
Then \[ c_{112}(E_{21}\lambda)=c_{112}(\lambda)\neq 0;\]
\[ c_{221}(E_{21}\lambda) = c_{221}(\lambda) - c_{112}(\lambda)= - c_{112}(\lambda) \neq 0;\]
\[ c_{113}(E_{21}\lambda) = c_{113}(\lambda) - c_{123}(\lambda)= - c_{123}(\lambda)\neq 0. \] 

\underline{Case 3}: $c_{223}(\lambda)= 0$ and $c_{123}(\lambda) \neq 0$.

Again we have $c_{112}(E_{21}\lambda)\neq 0$ and $c_{221}(E_{21}\lambda)\neq 0$, and this time
\[ c_{223}(E_{21}\lambda) = c_{223}(\lambda) + c_{123}(\lambda) = c_{123}(\lambda) \neq 0.\]
Hence we are done by Observation~\ref{obs:coeff2}.\\

\underline{Step 4}: There exists $g_4\in \SL_n(\dbZ)$ such that $c_{112}(g_4\lambda), c_{334}(g_4\lambda)\neq 0$.

By Step 3, we can assume that $c_{112}(\lam)\neq 0$, $c_{221}(\lam)\neq 0$ and $c_{113}(\lam)\neq 0$.
If $c_{443}(\lam)\neq 0$ or $c_{224}(\lambda)\neq 0$, we are done (by Observation~\ref{obs:coeff2}). 

So assume that $c_{443}(\lam)=c_{224}(\lam)=0$.
Then \[ c_{224}(E_{14}\lambda) = c_{224}(\lambda) - c_{221}(\lambda)= - c_{221}(\lambda) \neq 0.\]

\underline{Case 1:} $c_{413}(\lambda)= 0$. Then $c_{113}(E_{14}\lambda) = c_{113}(\lambda) + c_{413}(\lambda)\neq 0$, so
$c_{113}(E_{14}\lambda)\neq 0$ and $c_{224}(E_{14}\lambda) \neq 0$.

\underline{Case 2:} $c_{413}(\lambda) \neq 0$. Then $c_{443}(E_{14}\lambda) = c_{443}(\lambda) - c_{413}(\lambda) 
= - c_{413}(\lambda) \neq 0$ and $c_{221}(E_{14}\lambda)=c_{221}(\lambda)\neq 0$, 
so $c_{443}(E_{14}\lambda)\neq 0$ and $c_{221}(E_{14}\lambda) \neq 0$.

In both cases we are done by Observation~\ref{obs:coeff2}.\\

\underline{Final Step}. By Step~4, we can assume that $c_{112}(\lambda), c_{334}(\lambda)\neq 0$.
Set $\alpha=0$ if $c_{221}(\lam)\neq 0$ and $\alpha=1$ if $c_{221}(\lam)= 0$ and
$\beta=0$ if $c_{443}(\lam)\neq 0$ and $\beta=1$ if $c_{443}(\lam)= 0$.
Then $g=E_{21}^{\alpha}E_{43}^{\beta}$ satisfies the assertion of Lemma~\ref{lemkey4}.
\end{proof}

It is not hard to modify the above proof of Lemma~\ref{lemkey4} to obtain the following generalization:

\begin{Lemma}
\label{lemkey5}
Let $r\geq 2$ be an integer and $n \geq 2r$. Then for any nonzero $\lam\in\Lambda$ there exists $g\in \SL_n(\dbZ)$ such that
$c_{2i-1,2i-1,2i}(g\lam)$ and $c_{2i,2i,2i-1}(g\lam)$ are nonzero for all $1\leq i\leq r$.
\end{Lemma}

Using Lemma~\ref{lemkey5} instead of Lemma~\ref{lemkey4} in the proof of Theorem~\ref{thm:degree2}, it is easy to see that if $n\geq 6$, one can construct a $\rho$-CG sequence in $G=\IA_n$ by using only the relations (R1)-(R3). We will not provide the details of this argument. However, in \S~6
we will prove Theorem~\ref{thm:IAnrcg} which generalizes Theorem~\ref{thm:degree2} for $n\geq 12$ and whose proof will use even fewer relations,
namely only relations (R1)-(R3) with $a_i\neq b_j$ for all $i,j$.

\section{$\dbn$-groups and $\rho$-CG sequences}

\subsection{$\dbn$-groups}
\label{sec:ngps}

Recall that given $n\in\dbN$, we denote by $\dbn$ the set $\{1,2,\ldots, n\}$. 

\begin{Definition}\rm Let $n\in\dbN$. An {\it $\dbn$-group} is a group $G$ endowed with a collection of subgroups $\{G_I\}_{I\subseteq \dbn}$ such that
\begin{itemize}
\item[(i)] $G_{\dbn}=G$;
\item[(ii)] $G_I\subseteq G_{J}$ whenever $I\subseteq J$.
\end{itemize}
Given $d\in\dbN$, we will say that $G$ is {\it generated in degree $d$} if $G=\la G_I: |I|=d\ra$.

We will say that $G$ is a {\it commuting}\footnote{The authors thank Thomas Church and Andrew Putman for suggesting this terminology.} $\dbn$-group if $G_I$ and $G_J$ commute (elementwise) whenever $I$ and $J$ are disjoint.
\end{Definition}

Note that if $G$ is an $\dbn$-group, then any subgroup $H$ of $G$ becomes an $\dbn$-group if we set $H_I=H\cap G_I$. Clearly, if $G$ is commuting, then
so is $H$.

A simple example of a commuting $\dbn$-group is $G=\GL_n(\dbZ)$ with $\{G_I\}$ defined as follows.
Let $e_1,\ldots, e_n$ be the standard basic of $\dbZ^n$. Given $I\subseteq \dbn$, let $\dbZ^I=\oplus_{i\in I}\dbZ e_i$, and let 
\begin{equation}
\label{eq:SLnsubgps}
G_I=\{g\in G: g(\dbZ^I)\subseteq \dbZ^I\mbox{ and } g(e_j)=e_j\mbox{ for all }j\not\in I\}.
\end{equation}

Very similarly we can turn $G=\Aut(F_n)$ into a commuting $\dbn$-group. Choose a free generating set $x_1,\ldots, x_n$, and
for each $I\subseteq\dbn$ let $G_I$ be the subgroup consisting of automorphisms which leave the subgroup $\la x_i: i\in I\ra$ invariant and
fix $x_j$ for every $j\not\in I$. 

 In the sequel we will not encounter the $\dbn$-group structure on $\Aut(F_n)$ itself, but we will use the induced structure on $\IA_n$. Since $\IA_n$ is generated by the automorphisms $K_{ij}$ and $K_{ijk}$ defined in \S~\ref{sec:ian}, it is clear that $\IA_n$ is generated in degree $3$ as an $\dbn$-group. 

The fact that $\IA_n$ is a commuting $\dbn$-group captures some of the relations (R1)-(R3) from Lemma~\ref{lem:rel} (namely those with $a_i\neq b_j$
for all $i,j$) while the remaining relations (R1)-(R3) and the relations (R4) use more delicate properties of $\IA_n$ that seem hard to generalize to other groups.

As we will explain in \S~\ref{sec:torelli}, the Torelli group $\calI_n^1$ has an analogous structure of an $\dbn$-group generated in degree 
$3$, but that group is not {\it commuting}. Fortunately, $\calI_n^1$ is still {\it partially commuting} and {\it weakly commuting} according to the 
following definition (see \S~\ref{sec:torelli}).

\begin{Definition}\rm An $\dbn$-group $G$ is called 
\begin{itemize}
\item[(i)] \emph{partially commuting} if $G_I$ and $G_J$ commute whenever $I$ and $J$ are disjoint and $I$ consists of consecutive integers
(that is, $I=\{k,k+1,\ldots, k+|I|-1\}$ for some $k\in\dbn$);
\item[(ii)] \emph{weakly commuting} if for any disjoint $I$ and $J$ there exists $g\in G$ such that $G_I^g=g^{-1}G_I g$ commutes with $G_J$.
\end{itemize}
\end{Definition} 
The fact that $\calI_n^1$ is partially commuting will be essential for the proofs of the main theorems,  
while the weak commutativity of $\calI_n^1$ will play a secondary role (it will not be used in the proof of Theorem~\ref{OK} and could be eliminated from the proof of Theorem~\ref{coh_main}(c) at the expense of slightly increasing the lower bound on $n$). By contrast, in \cite{CEP} weak commutativity will play a central role, while partial commutativity will not be used at all.

\subsection{Constructing $\rho$-CG sequences in partially commuting $\dbn$-groups}
\label{sec:proofsketch}

In this subsection we fix an $\dbn$-group $G$ and an integer $m\geq 1$, and denote by $\Hom_m(G)$ the class of all homomorphisms $\rho$ from $G$ to another group such that $\rho(G)$ is {\it nilpotent of class $m$}, that is, $\gamma_{m+1} \rho(G)=\{1\}$ while $\gamma_{m} \rho(G)\neq \{1\}$
(equivalently, $\gamma_{m+1}G\subseteq \Ker\rho$, but $\gamma_{m}G\not\subseteq \Ker\rho$).
\vskip .07cm

Our goal is to find a sufficient condition on $G$ and $m$ that ensures the existence of a $\rho$-CG sequence for any $\rho\in \Hom_m(G)$. The precise formulation of this condition involves the notion of a regular $\SL_n(\dbZ)$-module and will be given in \S~\ref{sec:maintech}. We will use the proof of Theorem~\ref{thm:degree2} as a general guide to finding the suitable 
restrictions on $G$ and $m$.
\vskip .07cm

Let $L(m)=\gamma_m G/\gamma_{m+1}G$. As in \S~\ref{sec:proofian}, let $U$ be any abelian group which contains an isomorphic copy of every finitely generated abelian group, and let
$$\Lambda_m=\Hom_{\dbZ}(L(m),U).$$
Since $G$ is a finitely generated group, $L(m)$ is a finitely generated abelian group, whence
every subgroup of $L(m)$ is the kernel of some element of $\Lambda_m$. 

Now take any $\rho\in \Hom_m(G)$. 
Then $\gamma_{m+1}G\subseteq \Ker\rho$ while $\gamma_{m}G\not\subseteq \Ker\rho$, so
the subgroup $\Ker\rho\cap \gamma_m G$
contains $\gamma_{m+1}G$ and is strictly contained in $\gamma_m G$. Thus we can find a nonzero element $\lam_{\rho}\in \Lambda_m$
such that $$\Ker\lam_{\rho}=(\Ker\rho\cap \gamma_m G)/\gamma_{m+1} G.$$

For each $I\subseteq \dbn$ let $L(m)_I$ be the image of $\gamma_m G\cap G_I$ in $L(m)$.
The following technical definition will be very convenient for the subsequent discussion.

\begin{Definition}\rm Let $d\in\dbN$. An element $\lam\in\Lambda_m$ will be called
\emph{$d$-generic} if there exist disjoint subsets $I_1,\ldots, I_{d+1}$ of $\dbn$ each of which consists of consecutive integers 
such that $\lam$ does not vanish on $L(m)_{I_k}$ for each $1\leq k\leq d+1$.
\end{Definition}
\begin{Remark}\rm Note that whether a given $\lam\in\Lambda_m$ is $d$-generic is completely determined by $\Ker\lam$.
\end{Remark}

\begin{Lemma}
\label{lem:key}
Let $G$ be a partially commuting $\dbn$-group generated in degree $d$,
and let $\rho\in \Hom_m(G)$ be such that $\lam_{\rho}$ is $d$-generic.
Then $G$ admits a $\rho$-CG sequence.
\end{Lemma}
\begin{proof}
Since $\lam_{\rho}$ is $d$-generic, there exist elements $g_k\in \gamma_m G\cap G_{I_k}$ for $1\leq k\leq d+1$ such that $\rho(g_k)$ is non-trivial.
The assumption $\rho\in \Hom_m(G)$ guarantees that $\rho(\gamma_m G)=\gamma_m \rho(G)$ lies in the center of $\rho(G)$. Thus, $\rho(g_k)$ is central, and so $g_k\in C_{\rho}(G)$.

The hypotheses on $I_1,\ldots,I_{d+1}$ imply that $g_1,\ldots, g_{d+1}$ commute with each other. Now take any $I\subseteq \dbn$ with $|I|=d$. There exists $1\leq k\leq d+1$ such that $I\cap I_{k}=\emptyset$. Then every element of $G_I$ commutes with $g_k$. Since $G$ is generated in degree $d$, the
centralizers of the elements $g_1,\ldots, g_{d+1}$ generate $G$. Hence by Lemma~\ref{centr_trivial} $G$ has a $\rho$-CG sequence.
\end{proof}

Assume now that $G$ is a normal subgroup of a group $\Gamma$, and let $\overline\Gamma=\Gamma/G$. The conjugation action of $\Gamma$ on $\gamma_m G$ and $\gamma_{m+1}G$ induces an action
of $\overline\Gamma$ on $L(m)=\gamma_m G/\gamma_{m+1} G$. Define the corresponding  action of $\overline\Gamma$ on $\Lambda_m$ in the usual way (as in \S~\ref{sec:proofian}):
\begin{equation}
\label{eq:dualaction}
(g\lam) (x)=\lam(g^{-1}x)\mbox{ for all }g\in \overline\Gamma, \lam\in\Lambda_m\mbox{ and }x\in L(m).
\end{equation}

\begin{Claim} 
\label{rhoa}
Let $G$ be a partially commuting $\dbn$-group generated in degree $d$, and suppose that for every nonzero $\lam\in\Lambda_m$ the $\overline\Gamma$-orbit of $\lam$ contains a $d$-generic element. Then $G$ admits a $\rho$-CG sequence for any $\rho\in \Hom_m(G)$.
\end{Claim}
\begin{proof}
Take any $\rho\in \Hom_m(G)$. By assumption there exists $a\in \Gamma$ such that $\overline a\lam_{\rho}$ is $d$-generic, where $\overline a$ is the image of $a$ in $\overline\Gamma=\Gamma/G$. Define the homomorphism $\rho_a$ by $\rho_a(x)=\rho(a^{-1} xa)$. Then
\begin{multline*}
\Ker \lam_{\rho_a}=(\Ker\rho_a\cap \gamma_{m}G)/\gamma_{m+1}G=a(\Ker\rho\cap \gamma_{m}G) a^{-1}/\gamma_{m+1}G\\
=\overline a\left((\Ker\rho\cap \gamma_{m}G)/\gamma_{m+1}G\right)=\overline a\,\Ker\lam_{\rho}=\Ker(\overline a\lam_{\rho}).
\end{multline*}
Since $\overline a\lam_{\rho}$ is $d$-generic, so is $\lam_{\rho_a}$. By Lemma~\ref{lem:key}, $G$ admits a $\rho_a$-CG sequence and hence 
also admits a $\rho$-CG sequence.
\end{proof}

\subsection{How to apply Claim~\ref{rhoa}} 
\label{sec:topolya}
This subsection contains a somewhat informal discussion of how to ensure that the hypotheses of Claim~\ref{rhoa} hold. This discussion will not be formally used in any of the proofs in this paper, but it will help motivate the definition of a regular $\SL_n(\dbZ)$-module given in the next section. 
We will also mention one of the key ideas from \cite{CEP} used in the proof of the generalization of Theorem~\ref{thm:supermain} 
stated in \S~\ref{sec:closingcomments}.
\vskip .1cm

First, let us assume that for any permutation $\sigma\in S_n$ there exists $F_{\sigma}\in \overline\Gamma$ such that
\begin{equation}
\label{eq:fsigma}
F_\sigma(L(m)_I)=L(m)_{\sigma(I)}\mbox{ for every }I\subseteq\dbn.
\end{equation}
We did not formally record this property in the proof of Theorem~\ref{thm:degree2}, 
but it is closely related to Observation~\ref{kappa_action}(b)(ii).

Suppose now that $G$ is generated in degree $d$ with $n\geq md(d+1)$. An easy argument shows that $L(m)=\sum\limits_{|I|=md} L(m)_I$. Thus, given any nonzero $\lam\in\Lambda_m$, there exists at least one set $I$ with $|I|=md$ such that $\lam$ does not vanish on $L(m)_I$. Condition \eqref{eq:fsigma} then implies that for any $J$ with $|J|=md$ there exists $g\in \overline\Gamma$ such that $g\lam$ does not vanish on $L(m)_J$. At this point it is natural to ask if we can guarantee (under a suitable extra assumption) that there exists $g\in \overline\Gamma$ such that 
\begin{equation}
\label{eq:zariski}
g\lam \mbox{ does not vanish on } L(m)_J \mbox{ for {\bf every} } J \mbox{ with } |J|=md.
\end{equation}
Since $n\geq md(d+1)$, if such $g$ exists, then $g\lam$ is clearly $d$-generic. 

\vskip .12cm
There are two important cases where $g$ satisfying \eqref{eq:zariski} exists:

{\it Case 1:} $\lam=\lam_{\rho}$ for some homomorphism $\rho\in \Hom_m(G)$ which can be extended to a homomorphism defined on $\Gamma$. In this case $\rho_a$ defined by $\rho_a(x)=\rho(a^{-1} xa)$, with $a\in\Gamma$, has the same kernel as $\rho$ itself, and the proof of
Claim~\ref{rhoa} shows that $g\lam_{\rho}=\lam_{\rho}$ for all $g\in \Gamma/G$. Hence any $g\in \overline\Gamma$ satisfies \eqref{eq:zariski}.
This case, which may seem rather trivial, is actually sufficient for the proof of Theorem~\ref{FAb_concrete}(2), as we will explain in 
\S~\ref{sec:fabproof}.

{\it Case 2:} $\Im\lam$ is torsion-free. In this case we can embed $\Im\lam$ into the additive group of $\dbR$ (since $L(m)$ is finitely generated) and hence extend $\lam$ to a homomorphism $\lam_{\dbR}:L(m)\otimes \dbR\to \dbR$ such that $\lam(L(m)_I)\neq 0\iff \lam_{\dbR}(L(m)_I\otimes\dbR)\neq 0$. As will be shown in \cite{CEP}, one can find an element $g$ satisfying \eqref{eq:zariski} whenever there exists a subgroup $\overline\Gamma'$ of $\overline\Gamma$ which contains the elements $F_{\sigma}$ 
from \eqref{eq:fsigma} and whose image in $\GL(L(m)\otimes \dbR)$ has connected closure in the Zariski topology.  
The latter holds in the cases we are interested in: for instance, if $G=\IA_n$, we can take $\Gamma=\Aut(F_n)$, so that $\overline\Gamma=\GL_n(\dbZ)$ and let $\overline\Gamma'=\SL_n(\dbZ)$.
It is not hard to show that the closure of the image of $\SL_n(\dbZ)$ in $\GL(L(m)\otimes \dbR)$ is isomorphic to a quotient of $\SL_n(\dbR)$ and hence Zariski-connected. Note that the torsion-free case is sufficient for the proof of Theorem~\ref{OK}.

In order to prove Theorem~\ref{coh_main} and its consequences Corollary~\ref{cor:fgab} and Theorem~\ref{FAb_concrete}(1), 
we have to deal with $\lam$ not covered by either of the above cases, so we shall make more specific assumptions on the groups $\Gamma$ and $G$. 
We shall assume that
\begin{itemize}
\item[(i)] $\overline\Gamma$ contains a copy of $\SL_n(\dbZ)$;
\item[(ii)] $L(m)$ admits a graded structure $\{L(m)^I\}$ which is compatible with the filtered structure $\{L(m)_I\}$ in some sense (see condition~(iii) in Theorem~\ref{coh_main_new} and Lemma~\ref{claim1}); 
\item[(iii)] the graded structure $\{L(m)^I\}$ is ``well-behaved'' with respect to the action of the elements $E_{ij}, F_{ij}$ defined by \eqref{eij_def}.
\end{itemize}
Property (iii) will formally say that $L(m)$ is a regular $\SL_n(\dbZ)$-module, as defined in the next section. A nice feature of this definition is that we will only need to construct the grading and check (iii) for $m=1$; once this is done, we will automatically obtain the grading on $L(m)$
satisfying (iii) for arbitrary $m$. The assumptions (i)-(iii) will not be sufficient to ensure the existence of $g\in \SL_n(\dbZ)$ satisfying 
\eqref{eq:zariski}, but they will allow us to find $g$ such that $g\lam$ is $d$-generic, again assuming $n\geq md(d+1)$.

\section{Regular $\SL_n(\dbZ)$-modules}
\label{sec:reg}
 
\begin{Definition}\rm Let $V$ be an abelian group (written additively) and $n\in\dbN$. An {\it $\dbn$-grading of V} is a collection of additive subgroups
$\{V^I: I\subseteq \dbn\}$ of $V$ such that $$\sum_{I\subseteq \dbn}V^I=V$$ (note that the sum is not required to  be direct). The {\it degree} of $V$ (with respect to the grading $\{V^I\}$), denoted by $\deg(V)$, is the smallest integer $k$ such that $V^I=0$ whenever $|I|>k$ (this automatically implies that $\sum_{|I|\leq k} V^I=V$). 
\end{Definition}

Let $E_{ij},F_{ij}\in \SL_n(\dbZ)$ be defined by \eqref{eij_def} in \S~\ref{sec:ianab}.

\begin{Definition}\rm  Let $V$ be an $\SL_n(\dbZ)$-module endowed with an  $\dbn$-grading $\{V^I\}$. 
We will say that $V$ is {\it regular} if the following properties hold:
\begin{itemize}
\item[(1)] $E_{ij}$ acts trivially on $V^I$ if $I\cap\{i,j\}=\emptyset$.
\item[(2)] If $I\subseteq \dbn$, $i\in I$, $j\not \in I$, then for any $g\in \{E_{ij}^{\pm 1},E_{ji}^{\pm 1}\}$ and $v\in V^I$ we have
$$gv-v\in V^{I\cup\{j\}\setminus\{i\}}+V^{I\cup\{j\}}.$$
\item[(3)] If $I\subseteq \dbn$, $i,j\in\dbn$ with $i\neq j$, then $$F_{ij}V^I=V^{(i,j)I}$$ where $(i,j)I$ is the image of $I$
under the transposition $(i,j)$. In particular, if $i\in I$ and $j\not \in I$, then $F_{ij}V^I=V^{I\cup\{j\}\setminus\{i\}}$.
\end{itemize}
\end{Definition}

Informally speaking, the above definition generalizes some properties of the action of $\SL_n(\dbZ)$ on $\IA_n^{ab}$ that we used in the proof of Theorem~\ref{thm:degree2}. More precisely, (1), (2) and (3) correspond to parts (b)(i), (a) and (b)(ii) of Observation~\ref{kappa_action}, respectively.
There are variations of (2) that we could consider; we have chosen condition (2) above primarily because it is preserved under many natural operations (see Lemma~\ref{lem_regular} below). A key reason why we assume the existence of a graded structure (rather than a filtered structure)
in the above definition is that otherwise it seems hard to find a simple counterpart of (2).

\vskip .1cm

The following result provides our starting examples of regular $\SL_n(\dbZ)$-modules.

\begin{Lemma} 
\label{reg_chev}
Let $R$ be a commutative ring with $1$, $n\geq 2$, let $R^n$ be a free $R$-module of rank $n$ with basis
$e_1,\ldots, e_n$, and let $(R^n)^*=\Hom_R(R^n,R)$ be the dual module, with dual basis $e_1^*,\ldots, e_n^*$. Consider $R^n$ and $(R^n)^*$ as $\SL_n(\dbZ)$-modules with standard actions. Define $(R^n)^{\{i\}}=Re_i$,  $((R^n)^*)^{\{i\}}=Re_i^*$ for $1\leq i\leq n$ and $(R^n)^{I}=0$,
$((R^n)^*)^{I}=0$ for $|I|\neq 1$. Then with respect to these gradings $R^n$ and $(R^n)^*$ are regular
of degree $1$.
\end{Lemma}

\begin{proof} The only non-obvious part is condition (2) in the definition of a regular $\SL_n(\dbZ)$-module. Condition (2) is vacuous if
$|I|\neq 1$, so assume that $I=\{i\}$ and $j\neq i$, in which case $I\cup\{j\}\setminus\{i\}=\{j\}$ and $I\cup\{j\}=\{i,j\}$.

We have $E_{ij}^{\pm 1}(re_i)=re_i$, $E_{ji}^{\pm 1}(re_i)-re_i= \pm re_j\in (R^n)^{\{j\}}$,
$E_{ji}^{\pm 1}(re_i^*)=re_i^*$, $E_{ij}^{\pm 1}(re_i^*)-re_i^*= \mp re_j^*\in ((R^n)^*)^{\{j\}}$, so (ii) holds (note that in this case terms from $V^{I\cup\{j\}}$ do not arise).
\end{proof}

\begin{Lemma} 
\label{lem_regular}
The following hold:
\begin{itemize}
\item[(a)]  Let $V$ and $W$ be regular $\SL_n(\dbZ)$-modules. 
\begin{itemize}
\item[(i)] Define $(V\oplus W)^I=V^I\oplus W^I$. Then $V\oplus W$ is regular and $\deg(V\oplus W)=\max\{\deg(V),\deg(W)\}$.
\item[(ii)] Define $(V\otimes_{\dbZ} W)^I=\sum\limits_{I=I_1\cup I_2}V^{I_1}\otimes W^{I_2}$ (here the union $I_1\cup I_2$ need not be disjoint). Then
$V\otimes_{\dbZ} W$ is regular and $\deg(V\otimes_{\dbZ} W)\leq \deg(V)+\deg(W)$.
\end{itemize}
\item[(b)] Let $V$ be a regular $\SL_n(\dbZ)$-module and $U$ a submodule. For $I\subseteq \dbn$
set $U^I=U\cap V^I$ and $(V/U)^I=V^I+U$. Then $V/U$ is always regular with $\deg(V/U)\leq \deg(V)$,
and $U$ is regular with $\deg(U)\leq\deg(V)$ provided $U=\sum\limits_{I\subseteq \dbn}U^I$.
\item[(c)] Suppose that $\SL_n(\dbZ)$ acts on an $\dbN$-graded Lie ring $L=\oplus_{k=1}^{\infty}L(k)$ by graded automorphisms. Suppose that $L$
is generated by $L(1)$ (as a Lie ring) and that $L(1)$ is a regular $\SL_n(\dbZ)$-module of degree $d$. For each $k>1$ and $I\subseteq \dbn$
define 
\begin{equation}
\label{eq:reggrading}
L(k)^I=\sum\limits_{I=I_1\cup I_2\cup\ldots\cup I_k}[L(1)^{I_1}, L(1)^{I_2},\ldots, L(1)^{I_k}].
\end{equation}
Then $L(k)$ is a regular $\SL_n(\dbZ)$-module of degree at most $kd$.
\item[(d)]\footnote{This property as well as its application in the proof of Lemma~\ref{claim:trivial}(b) have been discovered during the writing of the follow-up paper \cite{CEP}.}
In the setting of (c) assume in addition that $[L(1)^I,L(1)^J]=0$ whenever $I$ and $J$ are disjoint. Then
for every $m\in\dbN$ the degree of $L(m)$ is at most $m(d-1)+1$.
\end{itemize}
\end{Lemma}
\begin{proof} Parts (a)(i) and (b) are straightforward. 

(a)(ii) As in the proof of Lemma~\ref{reg_chev}, we only need to check condition (2) in the definition of a regular module. So take
$I\subseteq\dbn$, $i\in I$ and $j\not\in I$. It is enough to check the condition for $v$ of the form $v=x\otimes y$ where
$x\in V^{I_1}$, $y\in W^{I_2}$ and $I_1\cup I_2=I$.
We will consider the case when $i$ belongs to both $I_1$ and $I_2$ (the case when $i$ only lies in one of those sets is analogous).
Thus, $I_1=K_1\cup\{i\}$ and $I_2=K_2\cup\{i\}$ with $i\not\in K_1, K_2$.

Let $g\in \{E_{ij}^{\pm 1}, E_{ji}^{\pm 1}\}$. Since $V$ and $W$ are regular, we have
$gx=x+x_1+x_2$ and $gy=y+y_1+y_2$ where $x_1\in V^{K_1\cup\{j\}}$, $x_2=V^{K_1\cup\{i,j\}}$, 
$y_1\in W^{K_2\cup\{j\}}$, $y_2=W^{K_2\cup\{i,j\}}$. Then $x_1\otimes y_1\in (V\otimes W)^{K_1\cup K_2\cup \{j\}}=(V\otimes W)^{I\cup\{j\}\setminus\{i\}}$
and each of the $7$ terms $x\otimes y_1$, $x\otimes y_2$, $x_1\otimes y$, $x_1\otimes y_2$, $x_2\otimes y$, $x_2\otimes y_1$ and $x_2\otimes y_2$ lies in $(V\otimes W)^{K_1\cup K_2\cup \{i,j\}}=(V\otimes W)^{I\cup\{j\}}$.
Hence $gv-v=gx\otimes gy-x\otimes y\in (V\otimes W)^{I\cup\{j\}\setminus\{i\}}+(V\otimes W)^{I\cup\{j\}}$, as desired.
\vskip .1cm

(c) Consider the map $\phi:L(1)^{\otimes k}\to L(k)$ given by 
$\phi(v_1\otimes\ldots\otimes v_k)=[v_1,\ldots, v_k]$ (where the commutator on the right-hand side is left-normed). Since $\SL_n(\dbZ)$ acts on $L$ by graded automorphisms, $\phi$ is a homomorphism of 
$\SL_n(\dbZ)$-modules, and since $L$ is generated by $L(1)$, $\phi$ is surjective.
Thus, $L(k)$ is a quotient of $L(1)^{\otimes k}$ as an $\SL_n(\dbZ)$-module, and the grading \eqref{eq:reggrading}
coincides with the quotient grading defined in (b). Thus, (c) follows from (a)(ii) and (b). 

(d) follows from the proof of (c) combined with the observation that
$[L(1)^{I_1},\ldots, L(1)^{I_m}]=0$ unless
$I_k$ has non-empty intersection with $I_1\cup I_2\cup\ldots\cup I_{k-1}$ for all $2\leq k\leq m$,
and the latter condition implies that $|\cup I_i|\leq (\sum |I_i|)-(m-1)$.
\end{proof}

\section{The main technical theorem}
\label{sec:maintech}
Given a group $G$, let $L(G)=\oplus_{i=1}^{\infty}\gamma_i G/\gamma_{i+1} G$. The graded abelian group $L(G)$ has a natural structure of a graded Lie ring with the bracket on homogeneous elements defined by
$$[g \gamma_{i+1}G, h \gamma_{j+1}G]= [g,h]\gamma_{i+j+1}G \mbox{ for all }g\in \gamma_i G \mbox{ and }h\in \gamma_j G,$$
where $[g,h]=g^{-1}h^{-1}gh$. The bracket operation is well defined since $[\gamma_i G,\gamma_j G]\subseteq \gamma_{i+j}G$.
Moreover, $\gamma_{i+1}G=[\gamma_i G,G]$ for each $i$, which implies that $L(G)$ is generated as a Lie ring by its degree $1$ component $G^{ab}=G/\gamma_{2}G$.

Suppose now we are given another group $\Gamma$ which contains $G$ as a normal
subgroup and a homomorphism $\phi:\SL_n(\dbZ)\to \Gamma/G$. Define the action of $\SL_n(\dbZ)$ on $L(G)$ by composing $\phi$ with the conjugation
action of  $\Gamma/G$ on $L(G)$. It is easy to see that the obtained action is by Lie ring automorphisms and explicitly given by
\begin{equation}
\label{eq:action}
g.(x\gamma_{i+1}G)=\left(\widehat{\phi(g)} x\widehat{\phi(g)}^{-1}\right)\gamma_{i+1}G\mbox{ for all } g\in \SL_n(\dbZ),\,\,\,i\in\dbN 
\mbox{ and }x\in\gamma_i G
\end{equation}
(where $\widehat{\phi(g)}$ is any lift of $\phi(g)$ to $\Gamma$).

The following is our main technical result which will be used to prove Theorems~\ref{OK} and Theorem~\ref{coh_main}.

\begin{Theorem} 
\label{coh_main_new}
Let $d,n\in\dbN$, and let $G$ be a finitely generated group with the structure of an $\dbn$-group generated in degree $d$.
Let $\Gamma$ and $\phi$ be as above, and consider the following four conditions:
\begin{itemize}
\item[(i)] $G$ is partially commuting (as defined in \S~\ref{sec:ngps}).
\item[(ii)] $L(1)=G/[G,G]=G^{ab}$ has the structure of a regular $\SL_n(\dbZ)$-module of degree at most $d$ (with respect to the action
of $SL_n(\dbZ)$ given by \eqref{eq:action}).
\item[(iii)] For every $I\subseteq\dbn$, with $|I|\geq d$, the image of $G_I$ in $G^{ab}=L(1)$ contains
$\sum\limits_{J\subseteq I}L(1)^{J}$.
\item[(iv)] $G$ is weakly commuting (as defined in \S~\ref{sec:ngps}).
\end{itemize}
Let $N\in\dbN$, and assume that either 
\begin{itemize}
\item (i)-(iii) hold and $n\geq d(d+1)(N-1)$ or 
\item (i)-(iv) hold and $n\geq (d+1)+(d^2-1)(N-1)$.
\end{itemize}
If $\rho$ is any non-trivial homomorphism from $G$ to another group with $\Ker(\rho)\supseteq \gamma_{N}G$, then $G$ has a $\rho$-CG sequence.
\end{Theorem}
\begin{Remark}\rm Since $L(1)^J=0$ for $|J|>d$ (by (ii)), it suffices to know (iii) when $|I|=d$. 
\end{Remark}

Before proving Theorem~\ref{coh_main_new}, we shall use it to prove Theorem~\ref{coh_main}(b). Theorem~\ref{coh_main}(b) is a direct consequence
of Proposition~\ref{cent_exist} and the following result.

\begin{Theorem}
\label{thm:IAnrcg} Let $N\geq 2$ and $n\geq 8N-4$ be integers and $G=\IA_n$. If $\rho$ is a non-trivial homomorphism from $G$ to another group
with $\Ker\rho\supseteq \gamma_N G$, then $G$ has a $\rho$-CG sequence.
\end{Theorem}

\begin{proof} We just need to check that the pair $(G,N)$ satisfies conditions (i)-(iv) in Theorem~\ref{coh_main_new} for $d=3$. We shall use the $\dbn$-group structure on $\IA_n$ defined in \S~\ref{sec:ngps}. Recall that this $\dbn$-group is commuting and generated in degree $3$, so (i) and (iv) hold.
Also recall that $L(1)=G^{ab}\cong V^*\otimes (V\wedge V)$ as $\SL_n(\dbZ)$-modules where $V=\dbZ^n$ with the standard $\SL_n(\dbZ)$-action. 
Using Lemmas~\ref{reg_chev} and \ref{lem_regular} we deduce that $L(1)$ is a regular $\SL_n(\dbZ)$-module of degree at most $3$, so (ii) holds. 
Moreover, if $\{e_i\}$ is the standard basis of $V$ and $\{e_i^*\}$ the dual basis, then 
for every $I\subseteq \dbn$ we have 
$$L(1)^I=\sum\limits_{I=\{i,j,k\}} \dbZ e_i^*\otimes e_j\wedge e_k \mbox{ where } i,j,k \mbox{ need not be distinct}.$$  
(so in particular $L(1)^I=0$ if $|I|\neq 2,3$).
Let us now take any $I$ with $|I|\geq 3$. The subgroup $G_I$ contains the element $K_{ij}$ for any distinct $i,j\in I$ and $K_{ijk}$ for any distinct $i,j,k\in I$. Since $K_{ij}$ and $K_{ijk}$ project to $e_i^*\otimes (e_i\wedge e_j)$ and $e_i^*\otimes (e_j\wedge e_k)$, respectively, we conclude that (iii) holds as well.
\end{proof}

Note that Theorem~\ref{thm:IAnrcg} yields not only Theorem~\ref{coh_main}(b), but also an alternative proof of Theorem~\ref{OK} for $G=\IA_n$, $n\geq 12$. By the same logic, to prove Theorem~\ref{OK} for the Torelli group and Theorem~\ref{coh_main}(c), we need to check conditions (i)-(iv) of
Theorem~\ref{coh_main_new} for $G=\calI_n^1$ and $d=3$. This verification requires more work and will be postponed till \S~\ref{sec:torelli}.

\vskip .12cm
The remainder of this section will be devoted to the proof of Theorem~\ref{coh_main_new}.

\subsection{Proof of Theorem~\ref{coh_main_new}}

Throughout this subsection we fix the notations introduced in Theorem~\ref{coh_main_new}. As before, we set $L(k)=\gamma_k G/\gamma_{k+1}G$ for all $k\in\dbN$.

Take any non-trivial homomorphism $\rho$ from $G$ to another group with $\Ker\rho\supseteq\gamma_N G$. Let $m$ be the nilpotency class of $\rho(G)$, so
that $1\leq m\leq N-1$. 

The proof of Theorem~\ref{coh_main_new} will start with a sequence of simple reductions; in fact, some steps in that sequence have already been
completed in \S~\ref{sec:proofsketch} when we established Claim~\ref{rhoa}. We will eventually reduce Theorem~\ref{coh_main_new} to 
Proposition~\ref{prop:tech}, a combinatorial statement about the regular $\SL_n(\dbZ)$-module $L(m)$, and it is only the proof of that proposition where the definition of a regular $\SL_n(\dbZ)$-module will be fully used. The main steps in the proof of Theorem~\ref{coh_main_new} are indicated in the following diagram:
\vskip .1cm

\centerline{
Proposition~\ref{prop:tech}$\Longrightarrow$ Claim~\ref{claim2}$\,\,\stackrel{\mbox{\tiny Lemma~\ref{lemma:cc}}}{\large \Longrightarrow}\,\,$ Claim~\ref{claim:filtered}$\,\,\stackrel{\mbox{\tiny Claim~\ref{rhoa}}}{\large \Longrightarrow}\,\,$ Theorem~\ref{coh_main_new}. 
}
\vskip .2cm

As in \S~\ref{sec:proofsketch}, for each $k\in\dbN$ let $L(k)_I$ be the image of $\gamma_k G\cap G_I$ in $L(k)$. 
Let $\{L(k)^I\}$ be the $\dbn$-grading on $L(k)$ obtained from the grading $\{L(1)^I\}$ via \eqref{eq:reggrading} (see Lemma~\ref{lem_regular}(c)).

Condition (iii) of Theorem~\ref{coh_main_new} simply says that 
$L(1)^J\subseteq L(1)_I$ whenever $J\subseteq I$ and $|I|\geq d$. We claim that the inclusion remains true if we replace $L(1)$ by $L(m)$: 

\begin{Lemma} For any subsets $J\subseteq I\subseteq\dbn$ with $|I|\geq d$ we have $L(m)^J\subseteq L(m)_I$.
\label{claim1}
\end{Lemma}
\begin{proof} By definition $L(m)^J$ is spanned by (some, not necessarily all) elements of the form $[x_1,\ldots, x_m]$
where $x_k\in L(1)^{J_k}$ for some $J_k\subseteq J$ (and thus $J_k\subseteq I$). Since $|I|\geq d$,  condition (iii) implies that 
$x_k\in L(1)_I$, so there exist exist elements $\{g_k\in G_{I}\}_{k=1}^d$ such that $x_k=g_k[G,G]$. Then $[x_1,\ldots, x_m]$
is the image in $L(m)$ of the element $[g_1,\ldots, g_m]$ which lies in $\gamma_m G_I\subseteq \gamma_m G\cap G_I$, so by definition
$[x_1,\ldots, x_m]\in L(m)_I$.
\end{proof}

Now set $M=1+m(d-1)$ if $G$ is weakly commuting and $M=md$ otherwise. 

\begin{Lemma}
\label{claim:trivial}
The following hold:
\begin{itemize}
\item[(a)] $n\geq (d+1)M$;
\item[(b)] $L(m)$ is a regular $\SL_n(\dbZ)$-module of degree at most $M$.
\end{itemize}
\end{Lemma}
\begin{proof} (a) holds by the hypotheses on $n$ in Theorem~\ref{coh_main_new} (since $m\leq N-1$).

(b) Recall that $L(1)$ is a regular $\SL_n(\dbZ)$-module of degree at most $d$, so by Lemma~\ref{lem_regular}(c),
$L(m)$ is regular of degree at most $md$. 

Suppose now that $G$ is weakly commuting, and let $I,J\subseteq\dbn$ be disjoint. By definition there exists
$g\in G$ such that $G_I^g$ and $G_J$ commute. Since for every $x\in G$ the elements $x$ and $x^g=x[x,g]$ have the same image in $G^{ab}=L(1)$, it follows that the images of $G_I$ and $G_J$ in $L(1)$ commute,
that is, $$[L(1)_I, L(1)_J]=0. \eqno (***)$$
We claim that $[L(1)^I,L(1)^J]=0$ for any disjoint $I,J$. Once this is proved, we deduce from Lemma~\ref{lem_regular}(d) that
$L(m)$ has degree at most $1+m(d-1)$. 

So let $I,J\subseteq \dbn$ be disjoint. If $|I|>d$ or $|J|>d$, then $L(1)^I=0$ or $L(1)^J=0$, and there is nothing to prove. If $|I|\leq d$ and  $|J|\leq d$, we can choose disjoint $I'$ and $J'$ with $|I'|=|J'|=d$, $I\subseteq I'$ and $J\subseteq J'$ (this is possible since $n\geq d^2+d\geq 2d$).
Then $[L(1)_{I'}, L(1)_{J'}]=0$ by (***); on the other hand $L(1)^I\subseteq L(1)_{I'}$ and  $L(1)^J\subseteq L(1)_{J'}$ by
condition (iii) in Theorem~\ref{coh_main_new} and hence $[L(1)^I,L(1)^J]=0$.
\end{proof}

As in \S~\ref{sec:proofsketch}, let $U$ be any abelian group containing a copy of every finitely generated abelian group and
$\Lambda_m=\Hom_{\dbZ}(L(m),U)$. Define the action of $\SL_n(\dbZ)$ on $\Lambda_m$ by
$$(g\lam) (x)=\lam(g^{-1}\!.\,x)\mbox{ for all }g\in \SL_n(\dbZ), \lam\in\Lambda_m\mbox{ and }x\in L(m)$$
where the action of $\SL_n(\dbZ)$ on $L(m)$ is given by \eqref{eq:action}. Equivalently, this action of $\SL_n(\dbZ)$ on $\Lambda_m$ is the composition
of $\phi:\SL_n(\dbZ)\to\Gamma/G$ with the action of $\Gamma/G$ on $\Lambda_m$ given by \eqref{eq:dualaction}. 
In particular, for every $\lam\in\Lambda_m$ its $\SL_n(\dbZ)$-orbit is contained in its $\Gamma/G$-orbit.

Thus, by Claim~\ref{rhoa}, Theorem~\ref{coh_main_new} is implied by the following statement:

\begin{Claim} 
\label{claim:filtered}
For any nonzero $\lam\in\Lambda_m$ there exists an element $\mu$ in the $\SL_n(\dbZ)$-orbit of $\lam$ and disjoint subsets 
$I_1,\ldots, I_{d+1}$ each of which consists of consecutive integers such that $\mu$ does not vanish on $L(m)_{I_k}$ for each $k$.
\label{rhoa1} 
\end{Claim}

\begin{Definition}\rm Let $\lam\in \Lambda_m$. Define $supp(\lam)$, the {\it support of $\lam$}, to be the set of all subsets $I\subseteq \dbn$ such that $\lam$ does not vanish on $L(m)^I$.
\end{Definition} 

Here is our next reduction.

\begin{Claim}
\label{claim2} For any nonzero $\lam\in\Lambda_m$ the $\SL_n(\dbZ)$-orbit of $\lam$ contains an element $\mu$ such that $supp(\mu)$ contains $d+1$ disjoint subsets.
\end{Claim}

The following lemma shows that Claim~\ref{claim:filtered} follows from Claim~\ref{claim2}.

\begin{Lemma} 
\label{lemma:cc}
Take any $\lam\in\Lambda_m$ whose support contains $d+1$ disjoint subsets. Then the $\SL_n(\dbZ)$-orbit of $\lam$ contains an element $\mu$ satisfying the conclusion of Claim~\ref{claim:filtered}.
\end{Lemma}
\begin{proof} 
For $1\leq k\leq d+1$ define $$I_{k}=\{(k-1)M+1,(k-1)M+2,\ldots,kM\}.$$
The sets $I_1,\ldots, I_{d+1}$ are disjoint and contained in $\dbn$ since $n\geq (d+1)M$. 

Let $J_1,\ldots, J_{d+1}$ be disjoint subsets in $supp(\lam)$.
By Lemma~\ref{claim:trivial} we have $|J_k|\leq M=|I_k|$. Hence there exists a permutation $\sigma\in S_n$
such that $\sigma(J_k)\subseteq I_k$ for each $k$. Write $\sigma$ as a product of transpositions $\sigma=\prod (i_t,j_t)$ and let 
$g=\prod F_{i_t j_t}$. By condition (3) in the definition of a regular module we have 
$$(g\lam)(L(m)^{\sigma(J_k)})=\lam(g^{-1}\!.\, L(m)^{\sigma(J_k)})=\lam(L(m)^{\sigma^{-1}(\sigma(J_k))})=\lam(L(m)^{J_k})\neq 0.$$
Since $|I_k|=M\geq d$ and $\sigma(J_k)\subseteq I_k$, Lemma~\ref{claim1} implies that $g\lam$ does not vanish on $L(m)_{I_k}$. Since $I_1,\ldots, I_{d+1}$ consist of consecutive integers, $\mu=g\lam$ has the desired property.
\end{proof}

Claim~\ref{claim2} easily follows from Proposition~\ref{prop:tech} below.

\begin{Definition}\rm Let $0\neq \lam\in\Lambda_m$ and $s\in\dbN$. If $I_1,\ldots, I_s$ are elements of $supp(\lam)$, define
$$D(I_1,\ldots, I_s)=\sum_{i=1}^s |I_i|-\sum_{i\neq j} |I_i\cap I_j|.$$
The maximum possible value of $D(I_1,\ldots, I_s)$ will be denoted by $D_s(\lam)$, and any $s$-tuple in $supp(\lam)$ on which this maximum
is achieved will be called {\it maximally disjoint for $\lam$}.
\end{Definition}

\begin{Proposition}
\label{prop:tech}
Fix $0\neq \lam\in\Lambda_m$ and $s\in\dbN$. Let $\mu$ be an element in the  $\SL_n(\dbZ)$-orbit of $\lam$ such that $D_s(\mu)$ is maximum possible,
and let $I_1,\ldots, I_s\in supp(\mu)$ be maximally disjoint for $\mu$. Then
one of the following holds:
\begin{itemize}
\item[(i)] $I_1,\ldots, I_s$ are disjoint;
\item[(ii)] $\cup I_i=\dbn$.
\end{itemize}
\end{Proposition}
Before proving Proposition~\ref{prop:tech}, we shall deduce Claim~\ref{claim2} from it. 

\begin{proof}[Proof of Claim~\ref{claim2}]
Let us apply Proposition~\ref{prop:tech} to $\lam$ (from the statement of Claim~\ref{claim2}) and $s=d+1$, and let $I_1,\ldots, I_{d+1}$
be as in the statement of  Proposition~\ref{prop:tech}. By Lemma~\ref{claim:trivial} we have $|I_k|\leq M\leq \frac{n}{d+1}$. If $I_1,\ldots, I_{d+1}$ are not disjoint, then $|\cup_{i=1}^{d+1} I_i|< (d+1)\cdot M\leq n$, so $\cup I_i\neq \dbn$, which contradicts Proposition~\ref{prop:tech}.
\end{proof}

We will now prove Proposition~\ref{prop:tech}, thereby completing the proof of Theorem~\ref{coh_main_new}.

\begin{proof}[Proof of Proposition~\ref{prop:tech}]
Assume that neither (i) nor (ii) holds. Without loss of generality we can assume that $|I_1\cap I_2|\neq\emptyset$, and choose
$i\in I_1\cap I_2$ and $j\not\in \cup_{l=1}^s I_l$ (with $j\in\dbn$). Since $(I_1,\ldots, I_s)$ is maximally disjoint for $\mu$, 
its support $supp(\mu)$ contains neither $I_k\cup\{j\}\setminus\{i\}$ nor $I_k\cup \{j\}$ for any $k$. Indeed, if $supp(\mu)$
contained one of those sets, replacing
$I_k$ by $I_k\cup\{j\}\setminus\{i\}$ or $I_k\cup \{j\}$ in $(I_1,\ldots, I_s)$ would produce an $s$-tuple in $supp(\mu)$ 
with a larger value of the function $D$, which is impossible.

Now take any $g\in \{E_{ij}^{\pm 1}, E_{ji}^{\pm 1}\}$. We claim that $supp(g\mu)$ contains $I_k$ for any $k$.
Indeed, by assumption $I_k\in supp(\mu)$, so $\mu(x)\neq 0$ for some $x\in L(m)^{I_k}$. We know that $j\not\in I_k$.
If $i\in I_k$, then by condition (2) in the definition of a regular module, 
$$g^{-1}x-x\in L(m)^{I_k\cup\{j\}\setminus\{i\}}+L(m)^{I_k\cup \{j\}}.$$
The latter containment is also true if $i\not\in I_k$ (in which case $g^{-1}x=x$ by condition (1)).
Since $supp(\mu)$ contains neither $I_k\cup\{j\}\setminus\{i\}$ nor $I_k\cup \{j\}$, we have $\mu(g^{-1}x-x)=0$.
Hence $(g\mu)(x)=\mu(g^{-1}x)=\mu(x)\neq 0$, so $I_k\in supp(g\mu)$.

Thus $supp(g\mu)$ contains $I_1,\ldots, I_s$. Since $(I_1,\ldots, I_s)$ is maximally disjoint for
$\mu$, we must have $D_s(g\mu)\geq D(I_1,\ldots, I_s)=D_s(\mu)$. On the other hand, 
$D_s(g\mu)\leq D_s(\mu)$ by the choice of $\mu$, so $D_s(g\mu)= D_s(\mu)$ and $(I_1,\ldots, I_s)$ is also maximally disjoint for
$g\mu$.

Applying the same argument to $E_{ij}\mu$, we conclude that $(I_1,\ldots, I_s)$ is also maximally disjoint for $E_{ji}^{-1}E_{ij}\mu$
and likewise maximally disjoint for $E_{ij}E_{ji}^{-1}E_{ij}\mu$. In particular, $I_1\cup\{j\}\setminus\{i\}\not\in supp (E_{ij}E_{ji}^{-1}E_{ij}\mu)$. On the other hand,
$E_{ij}E_{ji}^{-1}E_{ij}=F_{ij}$, and condition (3) in the definition of a regular module implies that
$supp(F_{ij}\mu)$ contains $(i,j)I_1=I_1\cup\{j\}\setminus\{i\}$, a contradiction.
\end{proof}

\section{Verifying the hypotheses of Theorem~\ref{coh_main_new} for the Torelli groups.} 
\label{sec:torelli}

Let $g\geq 0$, and let
$\Sigma_{g}^1$ be an orientable surface of genus $g$ with $1$ boundary component. The mapping class group $\Mod_{g}^1=\Mod(\Sigma_{g}^{1})$ is defined as the subgroup of orientation preserving homeomorphisms of $\Sigma_{g}^{1}$ which fix the boundary $\partial \Sigma_{g}^{1}$ pointwise modulo the isotopies which fix $\partial \Sigma_{g}^{1}$ pointwise. 

Choose a base point $p_0\in \partial \Sigma_g^1$.
Since $\Sigma_g^1$ has a bouquet of $2g$ circles as a deformation retract, the fundamental group $\pi=\pi_1(\Sigma_g^1, p_0)$ is free of rank $2g$.
The action of $\Mod_g^{1}$ on $\Sigma_g^1$ induces an action on $\pi$ and thus we obtain a homomorphism $\iota:\Mod_g^1\to\Aut(\pi)$. 
It is well known that $\iota$ is injective and $\Im\iota$ is equal to the stabilizer of the element $\prod\limits_{i=1}^g [x_{2i-1},x_{2i}]\in\pi$,
for a suitable choice of basis $\{x_1,\ldots, x_{2g}\}$ of $\pi$.
For each $k\in\dbN$ define $\calI_{g}^1(k)$ to be the kernel of the induced map $\Mod_g^1\to\Aut(\pi/\gamma_{k+1}\pi)$.

The filtration $\{\calI_{g}^1(k)\}_{k=1}^{\infty}$ is called the {\it Johnson filtration}, and the subgroups $\calI_{g}^1=\calI_{g}^1(1)$ and 
$\calK_{g}^1=\calI_{g}^1(2)$ are known as the {\it Torelli subgroup} and the {\it Johnson kernel}, respectively. 
Note that $\calI_{g}^1$ can also be defined as the set of all elements of $\Mod_{g}^1$ which act trivially
on the integral homology group $H_1(\Sigma_g^1)$.
\vskip .12cm
For the rest of this section we set $\Sigma=\Sigma_g^1$, $\calM=\Mod_g^1$, $\,\,\calI=\calI_g^1=\calI_g^1(1)$ and  $\calK=\calK_g^1=\calI_g^1(2)$. 
Starting in \S~\ref{sec:abBCJ}, we will assume that $g\geq 3$.

Our goal is to verify the hypotheses of Theorem~\ref{coh_main_new} for $\Gamma=\calM$, $G=\calI$, $n=g$ and $d=3$. Combined with Proposition~\ref{cent_exist}, this will prove Theorem~\ref{OK} for the Torelli group and Theorem~\ref{coh_main}(c) and thus will complete the proofs of Theorems~\ref{OK}~and~\ref{coh_main}.

The homomorphism $\phi:\SL_g(\dbZ)\to \calM/\calI$ to be used in Theorem~\ref{coh_main_new} will be defined in 
\S~\ref{sec:Torreg} (assuming the canonical isomorphism $\calM/\calI\cong \Sp(V)$ defined in \S~\ref{sec:spv}). 

\subsection{The  $\dbg$-group structure on $\calM$}
\label{section:CP}
We start by recalling the construction from \cite[\S 4.1]{CP}. It will be convenient to think of $\Sigma$ as a (closed) disk with $g$ handles attached; call the handles $H_1,\ldots, H_g$. Choose disjoint subsurfaces $X_1,\ldots, X_g$, each homeomorphic to $\Sigma_{1}^1$, such that $X_i$ contains $H_i$, and let $Y=\Sigma\setminus\cup_{i=1}^g {\rm Int}(X_i)$ where ${\rm Int}(X_i)$ is the interior of $X_i$. 
Fix a point $p\in {\rm Int}(Y)$, and for each $1\leq i\leq g$ choose a simple path $\delta_i$ in $Y$ which connects $p$ to a point on $\partial X_i$ such that the following hold
(see Figure~1):
\begin{itemize}
\item[(i)] the paths $\delta_i$ are disjoint apart from $p$;
\item[(ii)] if $C$ is a small circle around $p$, then the intersection points
$C\cap\delta_1,\ldots, C\cap \delta_g$ appear in the clockwise order.
\end{itemize}

\begin{figure}
\label{fig1}
\begin{tikzpicture}[scale=.5]

\draw[ultra thick]  (-0.5,0) ellipse (8 and 4);
\draw[ball color=black]  (-7,0) circle (.1);
\draw[ultra thick]  (-3,0) circle (1);
\draw[ultra thick]  (1,0) node (v2) {} circle (1);
\draw[ultra thick]  (5,0) circle (1);
\node (v1) at (-7,0) {};
\node  at (5.1,3.5) {$\partial \Sigma$};
\node  at (-7.1,-.5) {$p$};
\node  at (-1.2,-0.1) {$\partial X_1$};
\node  at (2.8,-0.1) {$\partial X_2$};
\node  at (6.8,-0.1) {$\partial X_3$};
\node  at (-5,0) {$\delta_1$};
\node  at (-5,-1.5) {$\delta_2$};
\node  at (-5,-2.5) {$\delta_3$};
\draw  plot[smooth, tension=.7] coordinates {(v1) (-3,-1)};
\draw  plot[smooth, tension=.7] coordinates {(v1) (-3,-2) (1,-1)};
\draw  plot[smooth, tension=.7] coordinates {(v1) (-3,-3) (5,-1)};
\end{tikzpicture}
\caption{}
\end{figure}

Now for each $I\subseteq \dbg$ define $S_I$ to be a closed regular neighborhood of 
$\cup_{i\in I}(\delta_i\cup X_i)$. Then $S_I$ is well defined up to isotopy (see Figure~2). 

\begin{figure}
\label{fig2}
\begin{tikzpicture}[scale=0.3]

\draw[ultra thick]  (0,0) ellipse (12 and 4);
\draw[ultra thick]  (-8,0) circle (1);
\draw[ultra thick]  (-4,0) circle (1);
\draw[ultra thick]  (0,0) circle (1);
\draw[ultra thick]  (4,0) circle (1);
\draw[ultra thick]  (8,0) circle (1);
\draw (-6,0) arc (0:90:-1.5);
\draw (-6,0) arc (0:90:1.5);
\draw (-9.5,0) arc (360:270:-1.5);
\draw (-9.5,0) arc (0:90:-2.5);
\draw[thin] (-8,1.5) -- (-7.5,1.5);
\draw[thin] (-7,-2.5) -- (3,-2.5);
\draw (3,-2.5) arc (270:360:2.5);
\draw (5.5,0) arc (0:90:1.5);
\draw[thin] (-4.5,-1.5) -- (-3.5,-1.5);
\draw (-3.5,-1.5) arc (270:360:1.5);
\draw (-2,0) arc (180:90:1.5);
\draw[thin] (-0.5,1.5) -- (4,1.5);
\node at (0,2.25) {$S_{134}$};

\end{tikzpicture}
\caption{The surface $S_{134}$ perturbed by an isotopy (only part of the surfaces inside $Y$ is shown)}
\end{figure}
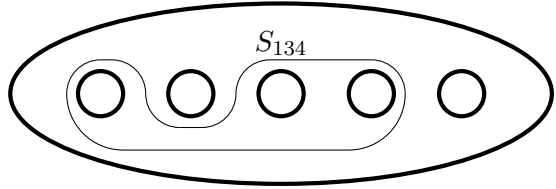
The following properties (1)-(4) are straightforward. Given two subsurfaces $R$ and $S$, we will say that certain relation between $R$ and $S$ holds up to isotopy if there exist subsurfaces isotopic to $R$ and $S$,
respectively, satisfying the stated relation:

\begin{itemize}
\item[(1)] $S_{\dbg}$ is isotopic to $\Sigma$;
\item[(2)] if $I\subseteq J$, then $S_I$ is contained in $S_J$ up to isotopy;
\item[(3)] if $I$ and $J$ are disjoint and $I$ consists of consecutive integers, then
$S_I$ and $S_J$ are disjoint up to isotopy;
\item[(4)] for any disjoint $I,J$ there exists a subsurface $S'_I$ homeomorphic to $S_I$ which is disjoint from $S_J$ and
such that $H_1(S_I)$ and $H_1(S'_I)$ have the same image in $H_1(\Sigma)$. 
\end{itemize}

Now define $\calM_I$ to be the subgroup of $\calM$ consisting of mapping classes which have a representative supported on $S_I$, and let $\calI_I=\calM_I\cap \calI$. Properties (1)-(3) above clearly imply that $\calM$ and $\calI$ are partially commuting $\dbg$-groups (here we use the obvious fact that homeomorphisms supported on disjoint surfaces commute). The fact that $\calI$ is weakly commuting follows from (4) and 
\cite[Lemma~4.2(ii)]{CP}.

A deep result of Church and Putman~\cite[Proposition~4.5]{CP} based on earlier work of Putman~\cite{Pu3} asserts
that if $g\geq 3$, then $\calI=\la \calI_I: |I|=3\ra$, so $\calI$ is generated in degree $3$ as a $\dbg$-group.
\vskip .1cm

In order to check the remaining conditions in Theorem~\ref{coh_main_new}, namely conditions (ii) and (iii), we need some preparations.

\subsection{$H_1(\Sigma)$ as an $\Sp_{2g}(\dbZ)$-module}
\label{sec:spv}
Let $V=H_1(\Sigma)$. Then $V$ is a free abelian group of rank $2g$ endowed with the canonical symplectic form $$([\alpha],[\beta])\mapsto
[\alpha]\cdot [\beta]$$ where $[\alpha]\cdot [\beta]$ is the algebraic intersection number between the closed curves $\alpha$ and $\beta$ on $\Sigma$. Clearly the action of $\calM/\calI$ on $V$ preserves this form, so there is a canonical group homomorphism
$\calM/\calI\to \Sp(V)$ where $\Sp(V)$ is the group of automorphisms of $V$ preserving the above form. It is well known that this homomorphism is an isomorphism, so from now on we will identify $\calM/\calI$ with $\Sp(V)$.

\vskip .12cm
For the rest of this section we will assume that $g\geq 3$.

In the next two subsections (\S~\ref{sec:abBCJ}, \S~\ref{sec:abfull}) we will describe the abelianization $\calI^{ab}$ 
as an $\Sp(V)$-module. This description was obtained by Johnson in \cite{Jo4} using earlier works of Birman and Craggs~\cite{BC}
and Johnson~\cite{Jo1,Jo2}.

\subsection{Johnson and Birman-Craggs-Johnson homomorphisms}
\label{sec:abBCJ}

Before describing $\calI^{ab}$, we need to introduce two smaller abelian quotients of $\calI$, the largest torsion-free quotient and the largest quotient of exponent $2$. These quotients are the images of the {\it Johnson homomorphism} (denoted by $\tau$) and the {\it Birman-Craggs-Johnson homomorphism}
(denoted by $\sigma$), respectively. 
We refer the reader to Johnson's papers \cite{Jo2} and \cite{Jo1} for the conceptual definition of $\tau$ and $\sigma$. For our purposes it will be sufficient to know the values of $\tau$ and $\sigma$ on the generating set for $\calI$ consisting of $1$-BP maps.

For each simple closed curve $\gamma$ on $\Sigma$ denote by $T_{\gamma}\in \calM$ the Dehn twist about $\gamma$. 
\begin{Definition}
\rm
Let $(\gamma,\delta)$ be a pair of disjoint non-separating simple closed curves on $\Sigma$. It is called a {\it bounding pair of genus $k$} if
\begin{itemize}
\item[(i)] $\gamma$ and $\delta$ are homologous to each other and non-homologous to zero;
\item[(ii)] the union $\gamma\cup \delta$ separates $\Sigma$; moreover, if $\Sigma_{\gamma,\delta}$
is the connected component of $\Sigma\setminus (\gamma\cup\delta)$ which does not contain $\partial \Sigma$, then
$\Sigma_{\gamma,\delta}$ has genus $k$.
\end{itemize}
If $(\gamma,\delta)$ is a bounding pair of genus $k$, then the map $T_{\gamma}T_{\delta}^{-1}$ always lies in $\calI$ and is called
a {\it $k$-BP map}.
\end{Definition}
Johnson~\cite{Jo0} proved that $\calI$ is generated by $1$-BP maps.

\vskip .1cm  
The {\it Johnson homomorphism} is a surjective homomorphism $\tau:\calI\to \wedge^3 V$ with $\Ker(\tau)=\calK$ (the Johnson kernel)
introduced in \cite{Jo2}. The induced map $\calI^{ab}\to \wedge^3 V$ (which we will also denote by $\tau$) is a homomorphism of $\Sp(V)$-modules. 
In \cite{Jo4}, Johnson proved that $\calI/\calK$ is the largest torsion-free abelian quotient of $\calI$.

We will now give an explicit formula for the values of $\tau$ on $1$-BP maps.
Let $(\gamma,\delta)$ be a bounding pair of genus $1$, and assume that $\gamma$ is oriented in such a way that $\Sigma_{\gamma,\delta}$ is on its left. Let $c\in V=H_1(\Sigma)$ be the homology class of $\gamma$, and choose any $a,b\in H_1(\Sigma_{\gamma,\delta})\subset H_1(\Sigma)$ such that $a\cdot b=1$. Then 
\begin{equation}
\label{johnhom}
\tau(T_{\gamma}T_{\delta}^{-1})=a\wedge b\wedge c.
\end{equation}
For the justification of \eqref{johnhom} and the related formula \eqref{bcjhom} below see \cite[\S~2]{Jo4} and references therein.

We now turn to the Birman-Craggs-Johnson homomorphism. Let $V_{\dbF_2}=V\otimes \dbF_2\cong H_1(\Sigma,\dbF_2)$.
Let $B$ be the ring of polynomials over $\dbF_2$ in formal variables $X=\{\overline v: v\in H_{\dbF_2}\setminus\{0\}\}$ subject to relations
\begin{itemize}
\item[(R1)] $\overline{v+w}=\overline v+\overline w+v\cdot w$ for all $\overline v,\overline w\in X$
\item[(R2)] $\overline v^2=\overline v$ for all $\overline v\in X$.
\end{itemize}
The group $\Sp(V)$ has a natural action on $B$ by ring automorphisms 
such that
$g(\overline v)=\overline{gv}$ for all $\overline v\in X$. Let $B_n$ be the subspace of $B$ consisting of elements representable by a polynomial of degree at most $n$. Then each $B_n$ is an $\Sp(V)$-submodule, and it is easy to see that $B_n/B_{n-1}\cong \wedge^n V_{\dbF_2}$ for each $n\geq 1$ and $B_0\cong \wedge^{0} V_{\dbF_2}$ (as $\Sp(V)$-modules). 

In \cite{Jo1}, Johnson constructed a surjective homomorphism 
$\sigma:\calI\to B_3$ which induces an $\Sp(V)$-module homomorphism $\calI^{ab}\to B_3$. Following \cite{BF}, we will refer to $\sigma$
as the {\it Birman-Craggs-Johnson (BCJ) homomorphism}. Here is the formula for $\sigma$ on $1$-BP maps:
\begin{equation}
\label{bcjhom}
\sigma(T_{\gamma}T_{\delta}^{-1})=\overline a\overline b (\overline c+1).
\end{equation}
where $a,b,c$ are defined as in \eqref{johnhom} except that this time they are mod $2$ homology classes.

\subsection{The full abelianization}  
\label{sec:abfull}

Let $\alpha: \wedge^3 V\to \wedge^3 V_{\dbF_2}$ be the natural reduction map, and let 
$\beta: B_3\to \wedge^3 V_{\dbF_2}$ be the unique linear map such that $\beta(B_2)=0$ and 
$\beta (\overline u\, \overline v\, \overline w)=u \wedge v\wedge w$ for all $u,v,w\in V_{\dbF_2}$. Clearly 
$\alpha$ and $\beta$ are both $\Sp(V)$-module homomorphisms. Let
\begin{equation}
\label{eq:W}
W=\{(u,v)\in \wedge^3 V\oplus B_3: \alpha(u)=\beta(v)\}.
\end{equation}

The following result is a reformulation of \cite[Thm~3]{Jo4}.

\begin{Theorem}[\cite{Jo4}]
\label{thm:fullab}
The map $(\tau,\sigma): \calI^{ab}\to \wedge^3 V\oplus B_3$ given by $g[G,G]\mapsto (\tau(g),\sigma(g))$ is injective
and $\Im((\tau,\sigma))=W$.
\end{Theorem}

Now let $I\subseteq \dbg$ be any subset with $|I|\geq 3$, and let $S_I$ be the corresponding subsurface of $\Sigma$
introduced in \S~\ref{section:CP}. Since $S_I$ is itself a closed orientable surface of genus $\geq 3$ with one boundary component, 
we can repeat the entire construction described in this section starting with $S_I$ instead of $\Sigma$, so in particular we can define the modules $V(I)=H_1(S_I)$ and $B_3(I)$ and the homomorphisms 
$\tau_I: \calI_I\to \wedge^3 V(I)$ and $\sigma_I: \calI_I\to B_3(I)$ (recall that $\calI_I$ was defined as the subgroup of $\calI$
consisting of mapping classes supported on $S_I$, but this group is canonically isomorphic to the Torelli subgroup of $\Mod(S_I)$ -- see 
\cite[\S~2.1]{Pu2}). 

\begin{Proposition} 
\label{prop:commutative}
The following diagrams are commutative:
\begin{equation}
\label{dia:rigid2}
\xymatrix{
(\calI_I)^{ab}\ar[r]^{\tau_I}\ar[d]_{} &\wedge^3 V(I)\ar[d]^{}&\qquad  & (\calI_I)^{ab}\ar[r]^{\sigma_I}\ar[d]_{}&B_3(I)\ar[d]^{}\\
\calI^{ab}\ar[r]^{\tau} &\wedge^3 V&\qquad & \calI^{ab}\ar[r]^{\sigma}&B_3
 }
\end{equation}
where the vertical maps are induced by the natural inclusions $\calI_I\to \calI$ and $H_1(S_I)\to H_1(\Sigma)$.
\end{Proposition}
\begin{proof} For both diagrams, it is enough to check commutativity for the values on 1-BP maps. This follows immediately from \eqref{johnhom} and \eqref{bcjhom} since if $(\gamma,\delta)$ is a genus $1$ bounding pair on $S_I$, it is also 
a genus $1$ bounding pair on $\Sigma$, and replacing $S_I$ by $\Sigma$ does not change the surface $\Sigma_{\gamma,\delta}$ or the homology classes $a,b$ and $c$ in  \eqref{johnhom} and \eqref{bcjhom}.
\end{proof}

\subsection{$\calI^{ab}$ as a regular $\SL_g(\dbZ)$-module}  
\label{sec:Torreg}
For each $1\leq i\leq g$ choose any basis $\{a_i,b_i\}$ for $H_1(S_i)\subset V$ s.t. $a_i\cdot b_i=1$.
Then  $V= \bigoplus\limits_{i=1}^g \left(\dbZ a_i\oplus \dbZ b_i\right)$, and $\{a_i,b_i\}_{i=1}^g$ is a symplectic basis for  
$V$, that is, $a_i\cdot a_j=b_i\cdot b_j=0$ for all $i,j$, and $a_i\cdot b_j=\delta_{ij}$.
Now define $\phi:\SL_g(\dbZ)\to \Sp(V)$ by 
\begin{equation}
\label{eq:slsp}
\phi(A)=\left(\begin{matrix} A& 0\\ 0&(A^{t})^{-1}\end{matrix}\right),
\end{equation}
where the above matrix is with respect to the ordered basis $(a_1,\ldots, a_g,b_1,\ldots, b_g)$.
Thus, all the $\Sp(V)$-modules defined in this section can also be considered as $\SL_g(\dbZ)$-modules. 

Note that the map
$(\tau,\sigma):\calI^{ab}\to W$ from Theorem~\ref{thm:fullab} is an isomorphism of $\SL_g(\dbZ)$-modules,
and the $\SL_g(\dbZ)$-module structure on $\calI^{ab}$ coincides with the one given by \eqref{eq:action} with $\phi$ as above 
(recall that we identified $\Sp(V)$ with $\calM/\calI$).

We could proceed working directly with $W$, but things can be simplified further thanks to the following observations.
Let $\pi:B_3\to  \oplus_{i=0}^3 \wedge^i V_{\dbF_2}$ be the unique linear map such that
$\pi(1)=1$, $\pi(\overline{x})=x $, $\pi(\overline{x}\,\overline{y})=x\wedge y$
and $\pi(\overline{x}\, \overline{y}\, \overline{z})=x\wedge y\wedge z$ where $x,y,z$ are distinct elements of the basis
$\{a_i, b_i\}_{i=1}^{g}$. Clearly, $\pi$ is bijective, and it is straightforward to check that
$\pi$ is an isomorphism of $\SL_g(\dbZ)$-modules 
\footnote{This is true since $\SL_g(\dbZ)$ preserves both $V_A=\oplus\dbZ a_i$ and $V_B=\oplus\dbZ b_i$ and the intersection form vanishes on both $V_A$ and $V_B$.}
(but not an isomorphism of $\Sp(V)$-modules!)

Now define $\pi': \wedge^3 V\oplus B_3\to \wedge^3 V\oplus\bigoplus\limits_{i=0}^3 \wedge^i V_{\dbF_2}$ by
$\pi'(u,v)=(u,\pi(v)-\alpha(u))$. Clearly, $\pi'$ is an isomorphism of $\SL_g(\dbZ)$-modules, and it is easy to show
that $\pi'(W)=\wedge^3 V\oplus\bigoplus\limits_{i=0}^2 \wedge^i V_{\dbF_2}$.
Thus, in view of Theorem~\ref{thm:fullab}, we have the following isomorphism of
$\SL_g(\dbZ)$-modules:
$$\lambda=\pi'\circ(\tau,\sigma): \calI^{ab}\to \wedge^3 V\oplus \left(\bigoplus_{i=0}^2 \wedge^i V_{\dbF_2}\right).$$
It is clear from \eqref{eq:slsp}  that $V$ is a direct sum of the natural $\SL_{g}(\dbZ)$-module $\dbZ^g$ and its dual.
Using Lemmas~\ref{reg_chev} and \ref{lem_regular} and the isomorphism $\lambda$, we endow 
$\calI^{ab}$ with the structure of a regular $\SL_g(\dbZ)$-module of degree $3$, so condition (ii)
in Theorem~\ref{coh_main_new} holds. 
Here is an explicit description of the obtained grading on $\calI^{ab}$.
Below $e_t$ stands for $a_t$ or $b_t$.
\begin{itemize}
\item[(1)] If $I=\{i<j<k\}$, then $(\calI^{ab})^I=\lambda^{-1}(\oplus \dbZ e_{i}\wedge e_{j}\wedge e_{k})$. 
\item[(2)] If $I=\{i<j\}$, then $(\calI^{ab})^I=\lambda^{-1}((\oplus \dbZ e_{i}\wedge a_j\wedge b_{j})\oplus
(\oplus \dbZ e_{j}\wedge a_i\wedge b_{i})\oplus(\oplus \dbF_{2} e_{i}\wedge e_{j}))$.
\item[(3)] If $I=\{i\}$, then $(\calI^{ab})^I=\lambda^{-1}(\dbF_2 a_i\oplus \dbF_2 b_{i})$.
\item[(4)] If $I=\emptyset$, then $(\calI^{ab})^I=\lambda^{-1}(\dbF_2)$.
\end{itemize}

It remains to check condition (iii), i.e., that the $\dbg$-group structure on $\calI$ defined in \S~\ref{section:CP} is compatible with the grading defined above. Let $I\subseteq \dbg$, with $|I|\geq 3$, and let $\iota_I: (\calI_I)^{ab}\to \calI^{ab}$
be the natural map. We need to show that 
\begin{equation}
\label{eq:inclusion}
\iota_I ((\calI_I)^{ab})\supseteq \sum\limits_{J\subseteq I}(\calI^{ab})^J
\end{equation}
or, equivalently, that $\lambda\circ \iota_I ((\calI_I)^{ab})$ contains $\sum\limits_{J\subseteq I}\lambda((\calI^{ab})^J)$.
We claim that the latter two sets are both equal to $Z_I:=\wedge^3 V(I)\oplus \left(\bigoplus\limits_{i=0}^2 \wedge^i V(I)_{\dbF_2}\right)$ (where we identify $V(I)=H_1(S_I)$ with its image in $V=H_1(\Sigma)$).
Indeed, $\sum\limits_{J\subseteq I}\lambda((\calI^{ab})^J)=Z_I$ directly from (1)-(3) above, while
the equality $\lambda\circ \iota_I ((\calI_I)^{ab})=Z_I$ follows from Proposition~\ref{prop:commutative}, as we now explain.

For simplicity of notation we identify $\wedge^3 V(I)$ and $B_3(I)$ with their canonical images in $\wedge^3 V$ and $B_3$.
With this identification, Proposition~\ref{prop:commutative} asserts that $(\tau_I,\sigma_I)=(\tau,\sigma)\circ \iota_I$ where each side 
of the equality is a map from $(\calI_I)^{ab}$ to $\wedge^3 V\oplus B_3$. Thus $$\lambda\circ \iota_I ((\calI_I)^{ab})=
\pi'\circ(\tau,\sigma)\circ \iota_I((\calI_I)^{ab})=\pi'\circ (\tau_I,\sigma_I)((\calI_I)^{ab}).$$ Since $|I|\geq 3$, we can apply
Theorem~\ref{thm:fullab} (with $\calI_I$ playing the role of $\calI$) to conclude that 
$(\tau_I,\sigma_I)((\calI_I)^{ab})=W_I$ where (using the notations of \S~\ref{sec:abfull}) 
$$W_I=\{(u,v)\in \wedge^3 V(I)\oplus B_3(I): \alpha(u)=\beta(v)\}=W\cap (\wedge^3 V(I)\oplus B_3(I)).$$ 
Thus,  $\lambda\circ \iota_I ((\calI_I)^{ab})=\pi'(W_I)$, and a straightforward computation shows that $\pi'(W_I)=Z_I$.

\begin{Remark}\rm The hypothesis $|I|\geq 3$ was needed twice in the proof of the inclusion \eqref{eq:inclusion} -- first to apply Proposition~\ref{prop:commutative} and then to use Theorem~\ref{thm:fullab}. We do not know if Proposition~\ref{prop:commutative} remains valid for all $I$, 
but the equality $(\tau_I,\sigma_I)((\calI_I)^{ab})=W_I$ and the inclusion \eqref{eq:inclusion} are definitely false at least when $|I|=0$
(that is, $I=\emptyset$) since $\calI_{\emptyset}$ is the trivial subgroup while $(\calI^{ab})^{\emptyset}=\lam^{-1}(\dbF_2)$ and $W_{\emptyset}=(0,\dbF_2)$. Thus it is essential that condition (iii) in Theorem~\ref{coh_main_new} is only required to hold when $|I|\geq d$ rather than for all $I$.
\end{Remark}

\section{The abelianization of the kernel of a homomorphism to a nilpotent group}
\label{sec:abel}

In this short section we prove Theorem~\ref{kernel_main} using the following results of Roseblade and Robinson. A major part of the argument below (including the use of Theorem~\ref{thm:Robinson}) was suggested to us by Andrei Jaikin who substantially simplified our original proof. Our original argument was inspired by the proof of \cite[Thm~3.6]{PS1}.

\begin{Theorem}[Roseblade] 
\label{thm:Roseblade}
Let $Q$ be a virtually polycyclic group. Then 
\begin{itemize} 
\item[(a)] Every simple $\dbZ[Q]$-module is finite and thus is an $\dbF_p[Q]$-module
for some prime $p$.
\item[(b)] Assume now that $Q$ is nilpotent, let $\Omega$ be the augmentation ideal of $\dbZ[Q]$
and $V$ a finitely generated $\dbZ[Q]$-module. If $\Omega M=0$ for every simple quotient $M$ of $V$,
then $\Omega^N V=0$ for some $N\in\dbN$.
\end{itemize}
\end{Theorem}
Parts (a) and (b) of Theorem~\ref{thm:Roseblade} are special cases of \cite[Cor A]{Ro} and \cite[Thm B]{Ro}, respectively.

\begin{Theorem}[Robinson \cite{Rb}] 
\label{thm:Robinson}
Let $Q$ be a nilpotent group, and let $M$ be a $Q$-module. Assume that either $H_0(Q,M)=0$ and $M$ is Noetherian or
$H^0(Q,M)=0$ and $M$ is Artinian. Then $H^i(Q,M)=0$ and $H_i(Q,M)=0$ for all $i>0$.
\end{Theorem}

We are now ready to prove Theorem~\ref{kernel_main}; in fact, we will prove a slightly more general statement:
\begin{Theorem}
\label{thm:16ext} 
Let $G$ be a finitely generated group, let $K$ be a normal subgroup of $G$ such that $Q=G/K$ is nilpotent,
and let $\Omega$ be the augmentation ideal of $\dbZ[Q]$.
 The following hold:
\begin{itemize}
\item[(a)] $K^{ab}=K/[K,K]$ is finitely generated as a $\dbZ[Q]$-module (equivalently, as a $\dbZ[G]$-module)
\item[(b)] $G/[K,K]$ is nilpotent if and only if $K/[K,K]$ is annihilated by some power of $\Omega$.
\item[(c)] Assume that for any finite field $F$ of prime order and any non-trivial finite-dimensional irreducible representation $\rho$
of $G$ over $F$ with $\Ker\rho\supseteq K$ we have $H^1(G,\rho)=0$. Then the equivalent conditions from (b) hold. 
\end{itemize}
\end{Theorem}
\begin{Remark} (a) and (b) are well known, but we are not aware of a reference in the literature.
\end{Remark} 
\begin{proof} We start with an observation that will be used for both (a) and (b). 
Given $c\in K$, let $\overline c=c[K,K]$ denote its image in $K^{ab}$. Then 
\begin{equation}
\label{eq:augaction}
(g^{-1}-1)\overline c=\overline{[c,g]}\mbox{ for all } c\in K \mbox{ and }g\in G
\end{equation}
(recall that $[c,g]=c^{-1}g^{-1}cg$).

(a) We know that $K$ contains $\gamma_n G$ for some $n$. 
Since $G$ is finitely generated, $K/\gamma_n G$ is also finitely generated (being a subgroup of a finitely generated nilpotent group).
Hence to prove that $K^{ab}$ is a finitely generated $\dbZ[G]$-module, it suffices to do the same for $(\gamma_n G)^{ab}$.
 
Let $X$ be a finite generating set for $G$. Then $\gamma_n G$
is generated as a group by left-normed commutators in $X$ of length at least $n$. But if $c$ is such a commutator and $x\in X$,
then $\overline{[c,x]}=(x^{-1}-1)\overline{c}$ by \eqref{eq:augaction}. Hence by straightforward induction $(\gamma_n G)^{ab}$ is generated 
as a $\dbZ[G]$-module by the images of left-normed commutators in $X$ of length exactly $n$.

\vskip .1cm
(b) Given $m\in\dbN$, we set $[K,{}_m G]=[K,\underbrace{G,\ldots, G}_{m \mbox{ times }}]$. First, we claim that $\Omega^m K^{ab}=\{0\}$ if and only if $[K,{}_m G]\subseteq [K,K]$.
The forward direction is immediate by \eqref{eq:augaction} while the opposite direction follows from \eqref{eq:augaction} and the fact that
$\Omega^m$ is generated as a left ideal of $\dbZ[Q]$ by elements of the form $(g_1-1)\ldots (g_m-1)$ with $g_i\in Q$.

Thus, (b) reduces to showing that $[K,{}_m G]\subseteq [K,K]$ for some $m$ if and only if
$G/[K,K]$ is nilpotent. If $G/[K,K]$ is nilpotent and $m$ is its nilpotency class, then $\gamma_{m+1}G\subseteq [K,K]$,
so in particular $[K,{}_m G]\subseteq [K,K]$. 

Conversely, suppose that $[K,{}_m G]\subseteq [K,K]$. Since $K\supseteq \gamma_{n}G$ for some $n$, we have $[K,K]\supseteq [\gamma_n G,\underbrace{G,\ldots, G}_{m \mbox{ times }}]=\gamma_{n+m}G$ and so
$G/[K,K]$ is nilpotent. 
\vskip .1cm

(c) We will prove that $\Omega^m K^{ab}=\{0\}$ for some $m\in\dbN$ by contradiction.
Let $V=K^{ab}$, and assume that $\Omega^m K^{ab}\neq \{0\}$ for any $m\in\dbN$.
By Theorem~\ref{thm:Roseblade}(b)
there is a simple $Q$-module $M$ which is a quotient of $V$ such that $\Omega M\neq 0$, so $Q$ acts on $M$ non-trivially.
Moreover, by Theorem~\ref{thm:Roseblade}(a), $M$ is a finite $F[Q]$-module for some finite field $F$ of prime order,
so $H^1(G,M)=0$ by the hypotheses of Theorem~\ref{kernel_main}.

Since $K$ acts trivially on $V$ and hence on $M$, we have a natural isomorphism
$$H^1(K,M)\cong \Hom(K,M)\cong \Hom(V,M).$$
Under  this isomorphism $H^1(K,M)^Q$, the subspace of $Q$-invariant elements of $H^1(K,M)$,
maps to $\Hom_Q(V,M)$, the subspace of $Q$-module homomorphisms from $V$ to $M$.
Since $M$ is a quotient of $V$ as a $Q$-module, we have $\Hom_Q(V,M)\neq 0$ and hence $H^1(K,M)^Q\neq 0$.

On the other hand, we have the inflation-restriction sequence
$$0\to H^1(Q,M)\to H^1(G,M)\to H^1(K,M)^Q\to H^2(Q,M)\to H^2(G,M).$$ 
Since $M$ is finite, simple and non-trivial, it is both Artinian and Noetherian, and both $H^0(Q,M)$ and  $H_0(Q,M)$ are trivial.
Thus, using either condition in Theorem~\ref{thm:Robinson}, we get
$H^1(Q,M)=H^2(Q,M)=0$, so $H^1(G,M)\cong H^1(K,M)^Q\neq 0$, a contradiction.
\end{proof}

\section{Abelianization of finite index subgroups}
\label{sec:fab}

\subsection{Proof of Theorem~\ref{FAb_concrete}} 
\label{sec:fabproof}
In this subsection we will prove Theorem~\ref{FAb_concrete} whose statement is recalled below.

\begin{Theorem1.9}
Let $G$ and $N$ be as in Theorem~\ref{coh_main}, 
and let $\Gamma=\Aut(F_n)$ if $G=\IA_n$
and $\Gamma=\Mod_g^1$ if $G=\calI_g^1$. Let $H$ be a subgroup of $\Gamma$ containing $\gamma_N G$.
The following hold:
\begin{itemize} 
\item [(1)] If $H$ is a finite index subgroup of $G$, then $\dim H^1(H,\dbC)=\dim H^1(G,\dbC)$.
\item [(2)] If $H$ is a finite index subgroup of $\Gamma$, then $H$ has finite abelianization.
\end{itemize}
\end{Theorem1.9}

We will deduce Theorem~\ref{FAb_concrete} from Theorem~\ref{coh_main} and the following two standard results.

\begin{Lemma} 
\label{lem:Shapiro}
Let $G$ be a group, $H$ a normal subgroup of $G$ of finite index and $Q=G/H$. Then
\begin{equation}
\label{eq:cohres}
H^1(H,\mathbb{C}) \cong H^1(G, \mathbb{C})\oplus \bigoplus\limits_{V \in Irr(Q)\setminus{V_0}} H^1(G, V)^{\dim(V)}
\end{equation}
where $Irr(Q)$ is the set of equivalence classes of irreducible complex representations of $Q$ and
$V_0$ is the trivial representation of $Q$. 
\end{Lemma}
\begin{proof}
By Shapiro's Lemma $H^1(H,\mathbb{C}) \cong H^1(G, {\rm Coind}_H^{G}(\mathbb{C}))$
where ${\rm Coind}_H^{G}(\mathbb{C})$ is the coinduced module.
Since $H$ is a normal finite index subgroup of $G$, we have an isomorphism of $G$-modules
${\rm Coind}_H^{G}(\mathbb{C})\cong \mathbb{C} [Q] \cong \bigoplus_{V \in Irr(Q)} V^{\dim(V)},$
which clearly implies Lemma~\ref{lem:Shapiro}. 
\end{proof}

\begin{Definition}\rm A representation $(\rho,V)$ of a group $G$ will be called \emph{finite-image} if $\rho(G)$ is finite.
\end{Definition}

\begin{Lemma} 
\label{lem:LubZhuk}
Let $G$ be a finitely generated group and $K$ a normal subgroup of $G$. The following are equivalent:
\begin{itemize}
\item[(i)] $H^{ab}$ is finite for any finite index subgroup $H$ of $G$ which contains $K$.
\item[(ii)] $H^1(G,V)=0$ for any irreducible finite-image complex representation $V$ of $G$ on which $K$ acts trivially.
\end{itemize}
\end{Lemma}
\begin{proof} First note that since $G$ is finitely generated, a finite index subgroup $H$ of $G$ is also finitely generated. Hence such $H$ has finite abelianization if and only if $H^1(H,\mathbb{C})=0$.

``(ii)$\Rightarrow$(i)'' Assume that (ii) holds, and let $H$ be a finite index subgroup of $G$ containing $K$. If $H$ is normal, then $H^1(H,\mathbb{C})=0$ directly by Lemma~\ref{lem:Shapiro}. In general, we can find a finite index subgroup $H'$ of $H$ which is normal in $G$
and still contains $K$. Then $H^1(H',\dbC)=0$. Since $H'$ has finite index in $H$, the restriction map
$H^1(H,\dbC)\to H^1(H',\dbC)$ is injective and therefore $H^1(H,\mathbb{C})=0$ as well.

``(i)$\Rightarrow$(ii)'' Now assume that (i) holds, and let $(\rho,V)$ be a representation satisfying the hypotheses of (ii). Then $H=\Ker\rho$ is a finite index normal subgroup of $G$ containing $K$, so $H^1(H,\dbC)=0$ by (i). Since $V$ is also an irreducible representation of $G/H$, Lemma~\ref{lem:Shapiro} implies that $H^1(G,V)=0$ as desired.  
\end{proof}

\begin{proof}[Proof of Theorem~\ref{FAb_concrete}] Part (1) follows directly from Theorem~\ref{coh_main} and Lemma~\ref{lem:Shapiro}.

Let us now prove (2). By Lemma~\ref{lem:LubZhuk} we need to show that $H^1(\Gamma,V)=0$ for any irreducible finite-image complex representation $V$ of $\Gamma$ on which $\gamma_N G$ acts trivially. The exact sequence of groups $1\to G\to \Gamma\to \Gamma/G\to 1$ yields the following inflation-restriction sequence:
\begin{equation}
\label{eq:infres}
0 \rightarrow H^1(  \Gamma/G, V^G) \rightarrow H^1(\Gamma, V) \rightarrow H^1(G, V)^{\Gamma/G}.
\end{equation}
Since $V$ is irreducible and $G$ is normal in $\Gamma$, we either have $V^G=0$ or $V^G=V$.

{\it Case 1: $V^G=0$}. Then $V$ is a direct sum of non-trivial irreducible finite-image complex representations of $G$ on which $\gamma_N G$ acts trivially. Thus $H^1(G, V)=0$ by Theorem~\ref{coh_main} and so $H^1(\Gamma, V)=0$ by \eqref{eq:infres}.

{\it Case 2: $V^G=V$}. In this case we can deduce that $H^1(\Gamma,V)=0$ directly from Lemma~\ref{lem:LubZhuk} and the previously known result
that $\Lambda^{ab}$ is finite for every finite index subgroup $\Lambda$ of $\Gamma$ containing $G$ (see \S~\ref{sec:fabremarks}). For completeness we will give 
an argument which does not use the latter result.

Recall that the group $\Gamma/G$ is isomorphic to either $\GL_n(\dbZ)$ or $\Sp_{2g}(\dbZ)$. It is well known that finite index subgroups of $\GL_n(\dbZ)$, $n\geq 3$, and $\Sp_{2g}(\dbZ)$, $g\geq 2$, have finite abelianization (for instance, because these groups have Kazhdan's property $(T)$ -- see, e.g., \cite[\S~1.3,1.4,1.7]{BHV}). Now we can apply Lemma~\ref{lem:LubZhuk} to $\Gamma/G$ with $K$ being the trivial subgroup to conclude that $H^1(  \Gamma/G, V)=0$. 

In view of \eqref{eq:infres}, it remains to prove that $H^1(G, V)^{\Gamma/G}=0$. Let $\Lambda$ be the kernel of the action of $\Gamma$ on $V$. By our assumptions $\Lambda$ is a finite index subgroup of $\Gamma$ containing $G$. Clearly, it suffices to show that $H^1(G, V)^{\Lambda/G}=0$.

Since $G$ acts trivially on $V$, the space $H^1(G,V)$ is a direct sum of finitely many copies
of $H^1(G,\dbC)$. Moreover, since $\Lambda$ acts trivially on $V$, each of those copies is $\Lambda/G$-invariant, and the action of $\Lambda/G$ on each copy is the standard action of $\Lambda/G$ on $H^1(G,\dbC)$. Thus, it suffices to show that $H^1(G, \dbC)^{\Lambda/G}=0$.

Let $\calG=\SL_n(\dbR)$ if $\Gamma=\Aut(F_n)$ and $\calG=\Sp_{2g}(\dbR)$ if $\Gamma=\Mod_g^1$.
Let $\overline\Gamma'=\SL_n(\dbZ)$ or $\Sp_{2g}(\dbZ)$, respectively.
The description of $G^{ab}$ given in \S~\ref{sec:ianab}, \S~\ref{sec:torelli} clearly implies that the natural action of $\overline\Gamma'$ on
$H^1(G,\dbC)\cong (G^{ab}\otimes \dbC)^*$ extends to a polynomial representation of $\calG$, and it is well known (and easy to check) that the obtained representation does not contain a copy of the trivial representation. Hence $H^1(G,\dbC)^{\calG}=0$. 

Now let $\overline\Lambda'=(\Lambda/G)\cap \overline\Gamma'$. Then $\overline\Lambda'$ is a finite index subgroup of $\overline\Gamma'$ 
and hence $\overline\Lambda'$ is Zariski dense in $\calG$ (one possible argument is that $\overline\Lambda'$ has infinite intersection with each of the root subgroups of $\calG$, so the Zariski closure of $\overline\Lambda'$ contains all the root subgroups which, in turn, generate $\calG$).
Thus, the equality $H^1(G,\dbC)^{\calG}=0$ implies that $H^1(G,\dbC)^{\overline\Lambda'}=0$, and so $H^1(G, \dbC)^{\Lambda/G}=0$ as well.
\end{proof}
\begin{Remark}\rm (a) The above argument shows that in order to prove Theorem~\ref{FAb_concrete}(2), we only need to know
Theorem~\ref{coh_main} for representations $(\rho,V)$ of $G$ which are extendable to $\Gamma$. The proof of Theorem~\ref{coh_main} in this special case
is much easier (see Case~1 in \S~\ref{sec:topolya}).

(b) As mentioned in the introduction, a positive answer to Question~\ref{q:FAb2} (for all $H$) would imply the same for Question~\ref{q:FAb}.
Indeed, let $G=\IA_n$ or $\Mod_g^1$ and assume that $\dim H^1(G,\dbC)=\dim H^1(G,\dbC)$ for every finite index subgroup $H$ of $G$. 
By Lemma~\ref{lem:LubZhuk}, $H^1(G,V)=0$ for every non-trivial irreducible finite-image complex representation $V$ of $G$. The proof of Theorem~\ref{FAb_concrete}(2) then shows that $H^1(\Gamma,V)=0$ for every irreducible finite-image complex representation $V$ of $\Gamma$, so applying Lemma~\ref{lem:LubZhuk} again we conclude that every finite index subgroup of $\Gamma$ has finite abelianization.
\end{Remark}

Let us now explain why the class of subgroups $H$ satisfying the hypotheses of Theorem~\ref{FAb_concrete} is quite large. Let $\Gamma$, $G$ and $N$ be as in Theorem~\ref{FAb_concrete}. The group $G/\gamma_N G$ is finitely generated nilpotent and hence residually finite. Thus, $G/\gamma_N G$ has plenty of finite index subgroups whose preimages in $G$ satisfy the hypotheses of Theorem~\ref{FAb_concrete}(1). 

For part (2) we can use, for instance, appropriate congruence subgroups. Recall that $\Gamma$ in Theorem~\ref{FAb_concrete} is either equal to $\Aut(F)$, where $F$ is free of rank $n$, or is isomorphic to a subgroup of $\Aut(F)$, where $F$ is free of rank $2g$. Given a normal subgroup $S$ of $F$, the corresponding congruence subgroup  $\Gamma(S)$ is defined to be the set of all $f\in \Gamma$ which leave $S$ invariant and act trivially on $F/S$. Note that if $S=\gamma_{k+1} F$, then $\Gamma(S)=G(k)$, the $k^{\rm th}$ term of the Johnson filtration. If $S$ has finite index in $F$, it is clear that $\Gamma(S)$ has finite index in $\Gamma$. If in addition $F/S$ is nilpotent of class at most $N$, then $\Gamma(S)$ contains $G(N)$ and hence also contains $\gamma_N G$.

\subsection{Earlier work on Question~\ref{q:FAb}}
\label{sec:fabremarks}

In this subsection we will summarize previously known positive results on Question~\ref{q:FAb} which is restated below:

\begin{Question1.7}
Let $\Gamma=\Aut(F_n)$ or $\Mod_g^1$. If $H$ is a finite index subgroup of $\Gamma$, does $H$ always have finite abelianization? 
\end{Question1.7}

We will start with the mapping class group case. Hain~\cite[Prop~5.2]{Ha} positively answered Question~\ref{q:FAb} for $\Gamma=\Mod_g^1$, $g\geq 3$ and $H$ containing the Torelli subgroup $\calI_g^1$. Putman~\cite[Thm~B]{Pu} generalized this result to all $H$ containing the Johnson kernel $\calI_g^1(2)$.\footnote{\cite[Thm~B]{Pu} actually covers all subgroups containing a ``large chunk'' of the Johnson kernel. The results of \cite{Ha} and \cite{Pu} apply to mapping class groups of surfaces with an arbitrary number of punctures and boundary components.}
Slightly earlier Boggi~\cite[Cor~3.11(b)]{Bo} proved by a different method that $H^{ab}$ is finite for all finite index subgroups of $\Mod_g$ 
containing $\calI_g(2)$. Finally, an interesting general approach to Question~\ref{q:FAb} was suggested by Putman and Wieland~\cite{PW}.

In the case $\Gamma=\Aut(F_n)$ a positive answer to Question~\ref{q:FAb} for $H$ containing $\IA_n$, $n\geq 3$, was given
by elementary arguments by Bhattacharya~\cite{Bh}~\footnote{Our argument in Case~2 of the proof of Theorem~\ref{FAb_concrete}(2)  is similar (but not identical) to the argument from \cite{Bh}.} and independently by Bogopolski and Vikentiev~\cite{BV}. This result also follows from the earlier work of Lubotzky and Pak~\cite{LP} (see Theorem~\ref{thm:LP} below).
\vskip .1cm

The following theorem was established in the Ph.D. thesis of Kielak: 

\begin{Theorem}\rm(\cite[Thm~5.3.2]{Ki}).\it
\label{thm:Kielak}
Let $n\geq 3$ and $H$ a finite index subgroup of $\Aut(F_n)$ such that $[H\cap \IA_n,H\cap \IA_n]$ contains $\gamma_k \IA_n$ for some $k$.
Then $H$ has finite abelianization.
\end{Theorem}

A very similar result follows from \cite[Thm~3.8]{LP} which asserts that $\Aut(F_n)/\IA_n(k)$ has Kazhdan's property $(T)$ for all $n\geq 3, k\geq 1$:

\begin{Theorem}\rm
\label{thm:LP}
Let $n\geq 3$ and $H$ a finite index subgroup of $\Aut(F_n)$ such that $[H,H]$ contains $\IA_n(k)$ (the $k^{\rm th}$ term of the Johnson filtration) for some $k$. Then $H$ has finite abelianization.
\end{Theorem}
\begin{proof} Let $\Gamma=\Aut(F_n)$ and $f:\Aut(F_n)\to\Aut(F_n)/\IA_n(k)$ the natural projection. Since $H$ has finite index in $\Gamma$, its image $f(H)$ has finite index in $f(\Gamma)$. Since $f(\Gamma)$ has property $(T)$, $f(H)$ has finite abelianization, whence $f([H,H])=[f(H),f(H)]$ has finite index in $f(\Gamma)$. Finally, since $[H,H]\supseteq \IA_n(k)=\Ker(f)$, we have 
$[H,H]=f^{-1}(f([H,H]))$, so $[H,H]$ has finite index in $\Gamma$ and thus $H^{ab}$ is finite.
\end{proof}

Corollary~\ref{cor:fgab} implies that Theorem~\ref{thm:Kielak} applies to any subgroup $H$ satisfying the hypotheses of Theorem~\ref{FAb_concrete}(2) with $\Gamma=\Aut(F_n)$. However this does not give a substantially new proof of Theorem~\ref{FAb_concrete}(2) for $\Gamma=\Aut(F_n)$ since the proof of Corollary~\ref{cor:fgab} also uses Theorem~\ref{coh_main}. It would be interesting to find other classes
of subgroups which can be explicitly constructed and satisfy the hypotheses of Theorem~\ref{thm:Kielak} or Theorem~\ref{thm:LP}.

\subsection{Some negative results for $\IA_3$}
\label{sec:fabia3}

In this final subsection we explain why the assertions of Theorem~\ref{coh_main} and both parts of Theorem~\ref{FAb_concrete} do not hold for $G=\IA_3$ and $N=2$. More explicitly, the following is true:

\begin{Proposition} 
\label{prop:ia3}
The following hold:
\begin{itemize}
\item[(a)] There exists a finite index subgroup $\Delta$ of $\Aut(F_3)$ such that $\Delta$ has infinite abelianization and
$[\IA_3:\Delta\cap \IA_3]=2$.
\item[(b)] There exists a subgroup $L$ of $\IA_3$ with $[\IA_3:L]=2$ (so in particular, $L\supseteq [\IA_3,\IA_3]$)
 such that $\dim H^1(L,\dbC)>\dim H^1(\IA_3,\dbC)$.
\item[(c)] There exists a one-dimensional complex representation $(\rho,V)$ of $\IA_3$ with $|\rho(\IA_3)|=2$ such that $H^1(\IA_3,V)\neq 0$.
\end{itemize}
\end{Proposition}
\begin{proof} (a) follows immediately from the work of Grunewald and Lubotzky~\cite[Prop~6.2]{GL} (see also \cite{BV}). We note that the existence of some finite index subgroup of $\Aut(F_3)$ with infinite abelianization was established earlier by McCool~\cite{Mc} by a different method.

(c) Let $\Delta$ be as in (a), and let $\Delta'=\IA_3\cdot\Delta$. Then $\Delta'$ is a finite index subgroup of $\Aut(F_3)$ containing $\IA_3$, so $\Delta'$ has finite abelianization by Theorem~\ref{thm:Kielak}~or~\ref{thm:LP}. Note that $[\Delta':\Delta]=[\IA_3:\Delta\cap \IA_3]=2$, so in particular $\Delta$ is normal in $\Delta'$. Applying Lemma~\ref{lem:Shapiro} with $G=\Delta'$ and $H=\Delta$, we get $H^1(\Delta,\dbC)\cong H^1(\Delta',\dbC)\oplus H^1(\Delta',V)$ where $(\rho,V)$ is the unique non-trivial irreducible complex representation of $\Delta'/\Delta$ 
(in particular $\dim V=1$ and $|\rho(\IA_3)|=|\rho(\Delta')|=2$). 

Since $\Delta'$ has finite abelianization while $\Delta$ does not, we have $H^1(\Delta',\dbC)=0$ and $H^1(\Delta,\dbC)\neq 0$,
whence $H^1(\Delta',V)\neq 0$. Since $V$ is irreducible and non-trivial, we have $V^{\IA_3}=V^{\Delta'}=0$. Thus, by inflation-restriction sequence $H^1(\Delta',V)$ embeds in $H^1(\IA_3,V)$, so $H^1(\IA_3,V)\neq 0$ and $(\rho,V)$ has the desired properties.

(b) Let $(\rho,V)$ be as in (c) and $L=\Ker\rho$. Then $(\rho,V)$ is the unique non-trivial irreducible complex representation of $\IA_3/L$. 
By Lemma~\ref{lem:Shapiro} with $G=\IA_3$ and $H=L$ we have $H^1(L,\dbC)\cong H^1(\IA_3,\dbC)\oplus H^1(\IA_3,V)$, which implies (b).
\end{proof}

\end{document}